\documentclass{article}
\usepackage{inputenc, amsmath}
\usepackage{amsthm}
\usepackage{xcolor}
\usepackage{amssymb, mathrsfs}
\usepackage{tabularx} 
\usepackage{booktabs} 
\usepackage{enumitem}  
\usepackage{lineno}
\usepackage[affil-it]{authblk}

\newtheorem{theorem}{Theorem}[section]
\newtheorem{defn}[theorem]{Definition}

\newtheorem{obs}[theorem]{Observation}
\newtheorem{fact}[theorem]{Fact}
\newtheorem{cor}[theorem]{Corollary}
\newtheorem{lemma}[theorem]{Lemma}

\newtheorem{claim}[theorem]{Claim}
\newtheorem{example}[theorem]{Example}

\renewcommand{\epsilon}{{\varepsilon}}

\providecommand{\keywords}[1]{\textbf{\textit{Key words.}} #1}
\usepackage{mathtools}
\usepackage{graphics}
\DeclarePairedDelimiter{\abs}{\lvert}{\rvert}

\title{Rainbow cycles in triangle-free graphs}
\author{Andrzej Czygrinow\thanks{School of Mathematical and Statistical Sciences, Arizona State University. Email: {\tt  aczygri@asu.edu}.} and 
Skand Parvatikar\thanks{School of Mathematical and Statistical Sciences, Arizona State University. Email: {\tt  sparvati@asu.edu}.}}

\usepackage{graphicx} 

\usepackage{geometry}
\date{December 2023}

\begin{document}

\maketitle
\begin{abstract}
Let $G = (V,E)$ be an edge-colored graph, and let $\delta^c(G) = \min_{v \in V} \{ d^{c}(v) \}$ where $d^c(v)$ is the number of colors on edges incident to a vertex $v$.  
We show that for a sufficiently large $n$ if $G$ is an edge-colored triangle-free graph of order $n$ that satisfies $\delta^c(G)\geq (n+7)/5$, then $G$ contains a rainbow cycle of length four, which improves a bound of Ding et al. and is best possible. In addition, we show that given $k$, there is $n_0$ such that for $n\geq n_0$, if $G$ is an edge-colored triangle-free graph with $\delta^c(G)> n/5+3$, then $G$ contains a rainbow cycle of length $4k$. 
\end{abstract}

\section{Introduction}
An \textit{edge-colored graph} is a graph $H = (V, E)$ with an edge-coloring $c: E \to \mathcal{C}$ where $\mathcal{C}$ is  a set of colors. A subgraph $F$ of $H$ is called \textit{rainbow} if all edges of $F$  have different colors. 
For a vertex $v$, let $d^c(v)$ be the number of colors on the edges incident to $v$, and let $\delta^c(H):=\min_{v\in V(H)} d^c(v)$ denote the minimum color-degree of $H$. H. Li \cite{H-Li} and independently Li et al. \cite{LNXZ} proved that if $H$ is an edge-colored graph on $n$ vertices that satisfies $\delta^c(H)\geq (n+1)/2$, then $H$ contains a rainbow triangle. For longer cycles, it is known (\cite{CMNO-even}, \cite{CMNO-odd}) that for every $k\in \mathbb{Z}^+$ there is $n_0$  such that if $H$ is an edge-colored graph of order  $n\geq n_0$  with $\delta^c(H)\geq (n+1)/2$, then $H$ has a rainbow cycle of length ${2k+1}$ and if $\delta^c(H)\geq (n+5)/3$ then $H$ has a rainbow cycle of length ${2k+2}$. The bounds for the minimum color degree are best possible { by considering a rainbow $K_{n,n}$ and a tail-coloring of a blowup of $C_3$, respectively}. In the case of balanced bipartite graphs of order $n$, if $H$ satisfies $\delta^c(H)\geq n/6+C$ for some constant $C$ then $H$ has a rainbow cycle of length four \cite{CY}. In the current paper we prove an analogous bound for triangle-free graphs.
\v{C}ada et al. \cite{Cada} showed that if $H$ is a triangle-free  graph of order $n$ with minimum color-degree at least $n/3+1$, then $H$ has a rainbow cycle of length four. This was substantially improved by Ding et al. \cite{Ding} who proved the following theorem.
\begin{theorem}\label{thm-ding}[Ding et al.]
    Let $H$ be an edge-colored triangle-free graph of order $n$ such that $\delta^c(H)\geq n/5 +3\sqrt{n}$. Then $H$ has a rainbow cycle of length four.
\end{theorem}
It is not difficult to see that the bound in Theorem \ref{thm-ding} is asymptotically best possible (Example \ref{ex-C4}). However, the bound for $\delta^c(H)-n/5$ can be further improved. In the current paper, we give such an improvement and  show the following.
\begin{theorem}\label{rainbowC4-thm}
There is $n_0$  such that for every $n\geq n_0$ the following holds. If  $H$ is an edge-colored { triangle-free} graph on $n$ vertices with $\delta^c(H)\geq (n+7)/5$, then $H$ contains a rainbow cycle of length four.
\end{theorem}
The proof relies on the regularity and stability  method, also used in \cite{CMNO-even}, and is split into two cases. In the first case, called the non-extremal case, we show that $H$ has $\Omega(n^3)$ rainbow triangles and in the second, the extremal case,  we analyze the structure that arises form the example showing tightness. Although the argument is based on the same general outline as in \cite{CMNO-even}, triangle-free graphs pose unique challenges that must be carefully addressed, and the extremal case in particular is quite delicate. The bound for the minimum color-degree in Theorem \ref{rainbowC4-thm} is best possible as seen from the following example.
\begin{example}\label{ex-C4}
Let $n \bmod 5 = 4$, and consider four pairwise disjoint sets $V_0,V_1, V_3, V_4$ each of size $(n+1)/5$ and one, $V_2$, of size $(n-4)/5$. Define the following digraph $D$. For every vertex $v\in V_i$ put the edge $vw$ for every $w\in V_{i+1}$. Now consider the edge-colored graph $H$ obtained from $D$ in two steps. First, for every edge $uv$ in $D$ color $uv$ with color $v$. Second, fix a perfect matching $M$ between $V_0$ with $V_1$ and  for every $vv'\in M$, recolor all edges of color $c_{v'}$ with  $c_v$, and recolor $vv'$ with $\alpha$ where $\alpha$ is a new color.
\begin{figure}[h] 
    \centering
    \includegraphics[width=0.9\textwidth]{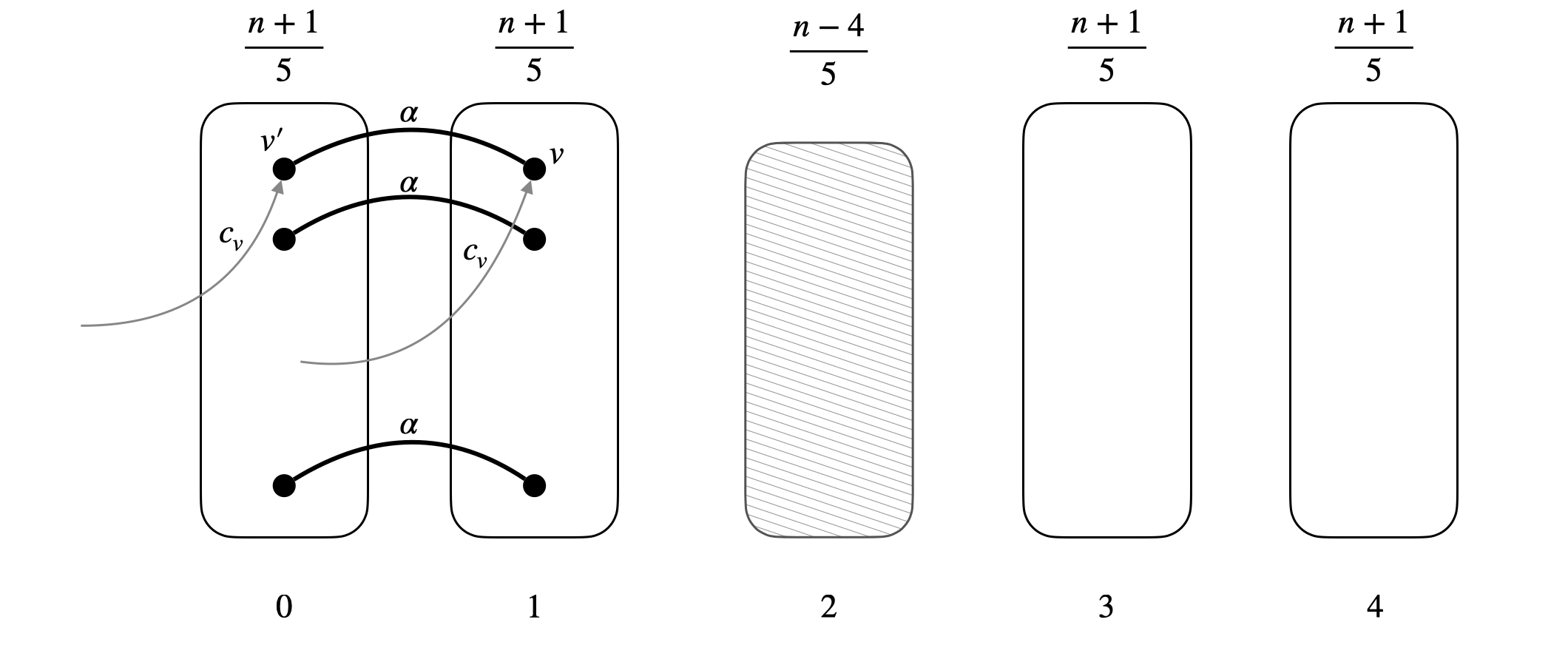} 
    \caption{Example \ref{ex-C4}}
    \label{fig:Lem-Example}
\end{figure}

If $v\in V_i$ for $i=0,2,3, 4$, then $v$ is incident to edges $vw$ of $(n+1)/5$ colors for each $w\in V_{i+1}$ and at least one edge of color $c_v$. Thus $d^c(v)=\frac{n+1}{5}+1$.
If $v\in V_1$, then $v$ is incident to edges $vw$ of $(n-4)/5$ colors for each $w\in V_2$, edges of color $c_v$, and color $\alpha$. Thus $d^c(v)=\frac{n-4}{5}+2$.
Thus $\delta^c(H)= \frac{n+6}{5}$.

Suppose $H$ has a rainbow cycle of length four. Then, since $D$ does not have directed cycles of length four, the cycle has exactly one edge of color $\alpha$.  Suppose $xvv'y$ is a rainbow $C_4$ with $v\in V_0$, $v'\in V_1$, and $\alpha =c(vv')$.  If $y\in V_2$, then $x\in V_1$ and so $c(xy)=c_y=c(v'y)$. Thus $y\in V_0$ and $c(yv')=c_{v'}$. If $x\in V_4$, then $c(xv)=c_v=c_{v'}=c(yv').$ Thus $x\in V_1$ and $c(vx)=c_x$. Then however $c(xy)\in \{\alpha, c_x\}$, and the cycle is not rainbow.

\end{example}
 The bound  in Theorem \ref{rainbowC4-thm} is not enough to guarantee rainbow cycles of any fixed even length. For example, for a rainbow cycle of length six, $\delta^c(H)> n/4$ is needed even when $H$ is bipartite. At the same time, the bound is almost enough to give any length that is a multiple of four. Specifically, we show that an edge-colored triangle-free graph $H$ of order $n$ contains a rainbow cycle of length $4k$ provided $\delta^c(H)> n/5+ 3$  and $n$ is large enough with respect to $k$. 
\begin{theorem}\label{rainbowgeneral-thm}
For every $k\in \mathbb{Z}^+$ there is $n_0\in \mathbb{Z}^+$  such that for every $n\geq n_0$ the following holds. If $H$ is an edge-colored triangle-free graph with $\delta^c(H)>n/5 +3$, then $H$ contains a rainbow $C_{4k}$.
\end{theorem}
On one hand, the proof of  Theorem \ref{rainbowgeneral-thm} is obtained by largely extending the methods for Theorem \ref{rainbowC4-thm}. On the other, additional observations are needed in the extremal case to take care of the remaining values of $4k \bmod 5 $.

The bound for $\delta^c(H)-n/5$  in Theorem \ref{rainbowgeneral-thm} is certainly not best possible. However,  the constant from Theorem \ref{rainbowC4-thm} does not seem to be enough to guarantee the existence of a rainbow cycle of length eight.
The rest of the paper is structured as follows. In Section \ref{prelim-section} we discuss preliminaries and  background related to the regularity lemma. Among others, we introduce the concept of a digraph associated with a colored graph $H$ that will be used to distingiush between the two cases mentioned before. In Section \ref{non-ext-section}, that consists of two parts, we discuss the non-extremal case.  The purpose of the first part of this section is to argue that we may assume that almost all vertices in the digraph associated with $H$ have the out-degree and in-degree approximately equal. The second part contains the proof of the fact that either $H$ contains desired rainbow cycles or the digraph has the special structure that arises from Example \ref{ex-C4}. Finally, in Section \ref{ext-section}, we analyze this special structure and using the bound for the minimum color-degree, we show that $H$ has a rainbow cycle of length $4k$.

\section{Preliminaries}\label{prelim-section}
A directed graph $D=(V, E)$ is a pair such that $V$ is a finite set and $E\subseteq \{(u,v)| u,v\in V, u\neq v\}$. A directed graph $D$ is called an {\it oriented graph} (or an {\it orientation}) if for every $u,v\in V(D)$ at most one of $uv, vu$ is in $E(D).$ 
A {\it directed cycle of length $l$} is an orientation of a cycle on $l$ vertices such that every vertex has its in- and out-degree equal to one. We say that an edge-colored graph $H$ has a {\it rainbow $C_l$} if  $H$ has a cycle of length $l$ with all edges of different colors. Similarly, we say that a directed graph $D$ has a {\it $C_l$}, if $D$ contains a directed cycle of length $l$. 
For two (not necessarily disjoint) sets of vertices $A, B$, let $E_D(A,B)$ denote the set of edges in $D$ that start at a vertex from $A$ and end at a vertex from $B$.

For a graph (or a digraph) $G=(V,E)$ we use $|G|:=|V|$ and $\| G\| = |E|$. 

\begin{defn}
Let $D$ be a 5-partite directed graph with vertex-partition $W_0,\dots, W_4$, let $\mu\in (0,1)$ and $n\in \mathbb{Z}^+$. We say that $D$  contains a $(n, \mu)$-blow-up of a directed cycle of length five if for every $i\in \{0,\dots, 4\}$, $(1-\mu)n \leq |W_i|\leq (1+\mu)n$  and $|E_D(W_i, W_{i+1})|\geq (1-\mu)|W_i|\cdot |W_{i+1}|.$  
\end{defn}
A directed graph $D$ is {\it triangle-free} if the underlying graph of $D$ has no cycle of length three. 
Let $H=(V,E)$ be an edge-colored graph. For a vertex $v$, let $N^c(v)$ denote a maximal set of vertices $w$ such that $w$ is a neighbor of $v$ in $H$ and all colors $c(vw)$ for $w\in N^c(v)$ are different. Define an {\it associated digraph $D_H$} as follows. Let $V(D_H):=V$ and let $E(D_H)$ contain all (directed) edges $vw$ for every $v\in V$  and  $w\in N^c(v).$  Note that $D_H$ depends on the choices of $N^c(v)$, and that it is not necessarily an orientation. We fix one choice of $D_H$ and record the following fact from \cite{CMNO-even}. 
\begin{fact}[\cite{CMNO-even}]\label{fact-di-cycles}
Let $l\in \mathbb{Z}^+$. If $H$ is an edge-colored graph on $n$ vertices that does not contain a rainbow cycle of length $l$ and $D_H$ is an associated digraph, then $D_H$ has at most $n^{l-1}$ directed cycles of length $l$.    
\end{fact}
Throughout the rest of the paper we suppose that $H$ is an edge-minimal counter-example subject to the bound on the minimum-color degree in Theorem \ref{rainbowC4-thm} (or Theorem \ref{rainbowgeneral-thm}).
Consequently, $H$ has not monochromatic path of length three. This observation will be used extensively in Section \ref{ext-section}.

The proofs of both theorems, Theorem \ref{rainbowC4-thm} and Theorem \ref{rainbowgeneral-thm}, will be split into non-extremal and extremal cases, based on the structural properties of $D_H$. In the non-extremal case, we use the reduced digraph of $D_H$ (defined below) and show that unless $D_H$ has special structure, the reduced digraph contains a directed cycle of length four. Then for a fixed $k\in \mathbb{Z}^+$, $D_H$ contains $\Omega(n^{4k})$ directed cycles of length $4k$ by simple counting, and so, in view of Fact \ref{fact-di-cycles}, $H$ contains a rainbow $C_{4k}$. The other option for $D_H$, as we will show in the next section, is to have the following structure. 
\begin{defn}\label{ext-def}Let $\eta>0$. 
We say that $D_H$ is $\eta$-extremal if there exist five pairwise-disjoint sets $V_0, V_1, V_2, V_3, V_4$ such that 
\begin{itemize}
\item $(1-\eta)n/5\leq |V_i|\leq (1+\eta)n/5$ for every $i=0, \dots, 4$, and
\item for every $i$ and every $v\in V_i$,
$$|N^+_{D_H}(v) \cap V_{i+1}|\geq (1-\eta)|V_{i+1}|,$$
$$|N^-_{D_H}(v) \cap V_{i-1}|\geq (1-\eta)|V_{i-1}|.$$
\end{itemize}
 \end{defn}
We continue with a discussion of the concepts that will be used to define the reduced digraph.
Let $D=(V,E)$ be a digraph and let $V_1, V_2$ be two disjoint non-empty subsets of $V$. The {\it density} of the ordered pair $(V_1,V_2)$ is 
\[d_D(V_1,V_2):=\frac{|E_D(V_1, V_2)|}{|V_1|\cdot |V_2|}.\]
We say that $(V_1,V_2)$ 
is \textit{$\epsilon$-regular with density $d$} if for every $V_1'\subseteq V_1$ and $V_2'\subseteq V_2$ with $|V_i'|\geq \epsilon |V_i|$, 
$|d_{D}(V_1', V_2')|-d|\leq \epsilon.$
A partition $V_0, V_1, \dots, V_t$ of $V$ is called \textit{$\epsilon$-regular} if $|V_0|\leq \epsilon |V|$, $|V_i|=|V_j|$ for $i,j\geq 1$, and all but at most $\epsilon t^2$ pairs $(V_i, V_j)$ with $i, j\geq 1$ are $\epsilon$-regular. 
A modification of the regularity lemma of Szemer\'{e}di \cite{szem} for directed graphs was obtained by Alon and Shapira \cite{alon-shapira}. 
\begin{lemma}[The digraph regularity lemma]\label{reg-lemma}
    Let $\epsilon \in (0,1)$ and let $t\in \mathbb{Z}^+$. Then there exist integers $T$ and $n_0$ such that every directed  graph $D=(V, E)$ on at least  $n_0$ vertices admits an $\epsilon$-regular partition $V_0, V_1,\dots, V_m$, with $t \leq m\leq T.$
\end{lemma}
Let $\eta>0$ be given and let $0<\epsilon \ll d\ll \eta$.  Consider an $\epsilon$-regular partition $V_0, V_1,\dots, V_m$ of $D$.
The {\it reduced digraph} of $D$ is the digraph $\mathcal{R}_D:=\mathcal{R}_D(\epsilon, d, t)$  with $V(\mathcal{R}_D)=\{1, \dots, m\}$ and the edge  set defined as follows. We have the edge from $i$ to $j$ if $(V_i, V_j)$ is $\epsilon$-regular and $d_D(V_i, V_j)\geq d$. From standard computations it follows that 
\begin{equation}\label{deltsR}\delta^+(\mathcal{R}_D)\geq (\delta^+(D_H)/n -O(d))m.\end{equation}

If $\mathcal{R}_D$ is not an oriented graph, then we have $ij\in E(\mathcal{R}_D)$ and $ji\in E(\mathcal{R}_D)$ for some $i,j\in V(\mathcal{R}_D)$. In that case, given a fixed $l$, the digraph $D_H[V_i, V_j]$ contains more than $n^{2l-1}$ directed cycles of length $2l$ provided $n$ is large enough with respect to $l$. Consequently, by Fact \ref{fact-di-cycles}, $H$ has a rainbow cycle of length $2l$. Therefore, we may assume that 
$\mathcal{R}_D$ is an oriented graph. 

One can show that if $\mathcal{R}_D$ contains a $(|\mathcal{R}_D|/5, \mu)$-blow-up of a directed cycle of length five, then $D_H$ is $\eta$-extremal. 
 \begin{lemma}\label{blow-up-extremal-lem}
 Let $n\in \mathbb{Z}^+$ be large enough. For every $0<\eta<1$, there is $0<\mu<1$ such that if $\mathcal{R}_D$ contains a $(|\mathcal{R}_D|/5, \mu)$-blow-up of a directed cycle of length five, then $D_H$ is $\eta$-extremal. 
 \end{lemma}
 The proof of Lemma \ref{blow-up-extremal-lem} is based on standard computations, where vertices that fail to satisfy the condition from Definition \ref{ext-def} are omitted. (Note that $V_0,\dots, V_4$ is not necessarily a partition.) 

\section{The Non-extremal Case}\label{non-ext-section}
In this section we address the non-extremal case. We suppose that $D$ is an oriented graph  with no directed cycle of length four, that satisfies  $\delta^+(D)\geq (1/5-\epsilon)|D|$ for some small  $\epsilon$. We show that $D$ contains a $(|D|/5, \mu)$-blow-up of a directed cycle of length five for some $\mu$ that depends on $\epsilon$. The argument splits into two parts. In the first part, we show that there is a sub-digraph $D'\subseteq D$ of order at least $(1-\epsilon)|D|$ such that for every $v\in V(D')$, 
$(1/5-\epsilon)|D'|\leq \abs{N^+_{D'}(u)}, \abs{N^-_{D'}(u)}\leq (1/5+\epsilon)|D'|$. In the second part, we argue that $D'$ contains a $(|D'|/5, \mu)$-blow-up of a directed cycle of length five. This gives a $(|D|/5, O(\mu))$-blow-up of a directed cycle of length five in $D$, and so by Lemma \ref{blow-up-extremal-lem}, $D$ is $\eta$-extremal with $\eta$ that depends on $\mu$. 
\subsection{Part 1}
 We will start with the following observation.
\begin{lemma}\label{non-ext-lem01}Let $\xi>0$ and let $D$ be an triangle-free oriented graph such that for every $v\in V(D)$, $|N^+_D(v)|= (1/5-\xi)\abs{D}$ and  such that there exists an edge $u_0v_0$ in $D$ that satisfies $|N^-_D(u_0)|+|N^-_D(v_0)|\geq (2/5+ 4\xi)\abs{D}$. Then $D$ has a directed cycle of length four.
\end{lemma}
\begin{proof}
Let $A_0:=N^-_D(u_0)$, $B_0:=N^+_D(u_0)$, $A_1:=N^-_D(v_0)$, $B_1:=N^+_D(v_0)$, and note that $|B_0|=  |B_1| = (1/5-\xi)|D|$. Since $D$ is triangle-free, all sets $A_0, A_1, B_0, B_1$ are pairwise disjoint.
Let $C:=V(D)\setminus (A_0\cup A_1\cup B_0\cup B_1),$ and note that
\begin{equation}\label{non-ext-p1-eq0}
|C|\leq (1/5- 2\xi)\abs{D}.
\end{equation}
 In addition, using the fact $D$ is triangle-free, $D[A_0\cup B_0]$, $D[A_1\cup B_1]$ have no edges. 
 If $E_D(B_1,A_0)\neq \emptyset$, then there is a directed  cycle of length four in $D$. Thus we may assume that for every $x\in B_1$, $N^+_D(x)\subseteq B_0\cup C.$ Similarly, if $y\in N^+_D(B_1)\cap B_0$, then we may 
 assume that $N^+_D(y)\subseteq B_1\cup C$. 
 
 Let $G$ be the sub-digraph of $D$ induced by $(N^+_D(B_1) \cap B_0) \cup B_1$. 
Since $D$ is triangle-free, if $xy\in E(G)$, then  $|N^+_D(x)\cap C|+ |N^+_D(y)\cap C|\leq |C|$. Now, in view of (\ref{non-ext-p1-eq0}), for $xy\in E(G)$ 
\begin{equation}\label{eq-non-ext-p1-eq1}
|N^+_G(x)|+|N^+_G(y)|\geq 2(1/5-\xi)\abs{D}-|C|\geq \abs{D}/5.
\end{equation}
Consequently,
\begin{equation}\label{eq-non-ext-p1-eq2}\sum_{xy\in E(G)} |N^+_G(x)|+|N^+_G(y)| \geq \abs{D}\cdot \|G\|/5.\end{equation}

At the same time, for $i=0,1$ and $w\in V(G)\cap B_i$, $|N^+_G(w)|+|N^-_G(w)|\leq |B_{1-i}|=(1/5-\xi)|D|$. Therefore,
\begin{equation}\label{eq-non-ext-p1-eq3}\sum_{xy\in E(G)} |N^+_G(x)|+|N^+_G(y)|=\sum_{w\in V(G)}(|N^+_G(w)|\cdot |N^+_G(w)|+|N^+_G(w)|\cdot |N^-_G(w)|)\leq (1/5-\xi)\abs{D} \sum_{w\in V(G)} |N^+_G(w)|.\end{equation}
 Consequently, from (\ref{eq-non-ext-p1-eq2}) and (\ref{eq-non-ext-p1-eq3}),
$||G||/5 \leq (1/5-\xi)||G||$ which is not possible since $||G||>0$.

\end{proof}
We will now show that for every vertex $v$ in $D$  but a small number of exceptions, we may assume $|N^-_D(v)|\geq (1/5-\mu)\abs{D}$ for some $\mu =\mu(\xi)$ with $\mu \rightarrow 0$ when $\xi\rightarrow 0$.
To prove this observation, we will use the following ``defect form" of Cauchy-Schwarz inequality \cite{szem}.
\begin{lemma}\label{CS-ineq}
Let $p,q\in \mathbb{Z}^+$ with $q\leq p$, $\eta>0$, and let $x_1, \dots, x_p$ be non-negative real numbers. If $\sum_{k=1}^qx_k=\frac{q}{p}\sum_{k=1}^px_k+\eta$, then
$$\sum_{k=1}^px_k^2\geq \frac{1}{p}\left(\sum_{k=1}^px_k\right)^2+\frac{\eta^2p}{q(p-q)}.$$
\end{lemma}
Using Lemma \ref{CS-ineq} we can prove the following.
\begin{lemma}\label{lem2-non-ext-part1}
    For every $\mu>0$ there exists $\xi>0$ such that the following holds. If $D$ is an oriented triangle-free graph  such that for every $v\in V(D)$, $|N^+_D(v)|=(1/5-\xi)\abs{D}$ and for every edge $uv$ in $D$, $|N^-_D(u)|+|N^-_D(v)|\leq (2/5+ 4\xi)\abs{D}$, then the number of vertices $v$ in $D$ such that $|N^-_D(v)|\geq (1/5+\mu)\abs{D}$ is at most $\mu \abs{D}$.    
\end{lemma}
\begin{proof}Let $U:= \{v\in V(D) : |N^-_D(v)|\geq (1/5+\mu)\abs{D}\}$ and suppose $|U|> \mu \abs{D}$.
We have
\begin{equation}\label{eq-non-ext-p1-eq3.5}\|D\|=\sum_{v\in V(D)}|N^-_D(v)|=\sum_{v\in V(D)}|N^+_D(v)|=(1/5-\xi)\abs{D}^2\end{equation}
and so
$$\sum_{u\in U}|N^-_D(u)|\geq (1/5-\xi)\abs{D} \cdot |U|+ (\mu+\xi)\abs{D}\cdot |U|=\frac{|U|}{\abs{D}}\sum_{v\in V(D)}|N^-_D(v)|+(\mu+\xi)\abs{D}\cdot |U|.$$
Therefore, by Lemma \ref{CS-ineq} with $\eta=(\mu+\xi)|D|\cdot |U|, p = |D|, q = |U|$,
$$\sum_{v\in V(D)}|N^-_D(v)|^2\geq \frac{1}{\abs{D}}\left(\sum_{v\in V(D)}|N^-_D(v)|\right)^2+ \frac{(\mu+\xi)^2\abs{D}^3|U|^2}{|U|(\abs{D}-|U|)}\geq \frac{1}{|D|}\left(\sum_{v\in V(D)}|N^-_D(v)|\right)^2+\mu^3\abs{D}^3,$$
since $|U|\geq \mu \abs{D}$. Consequently,
\begin{equation}\label{eq-non-ext-p1-eq4}
 \sum_{v\in V(D)}|N^-_D(v)|^2\geq\frac{1}{\abs{D}}\|D\|^2+\mu^3\abs{D}^3= (1/5-\xi)\abs{D}\cdot \|D\|+\mu^3\abs{D}^3.   
\end{equation}
Since for every $v$, $|N^+_D(v)|\geq (1/5-\xi)\abs{D}$, we have
$$\sum_{uv\in D}|N^-_D(u)|+|N^-_D(v)|= \sum_{w\in V(D)}|N^-_D(w)|(|N^-_D(w)|+|N^+_D(w)|)= (1/5-\xi)\abs{D}\cdot \|D\|+ \sum_{w\in V(D)}|N^-_D(w)|^2.$$
Thus, using the assumption that for $uv\in D$, $|N^-_D(u)|+|N^-_D(v)|\leq (2/5+4\xi)|D|$, by (\ref{eq-non-ext-p1-eq4}),
$$(2/5+4\xi)\abs{D} \cdot \|D\|\geq  (1/5-\xi)\abs{D}\cdot\|D\|+ \sum_{w\in V(D)}|N^-_D(w)|^2 \geq (2/5-2\xi)\abs{D}\cdot\|D\|+\mu^3\abs{D}^3,$$
which is a contradiction when $\xi < \mu^3/6.$
\end{proof}
From the previous two lemmas, we obtain the following corollary.
\begin{cor}\label{cor-NE-part1}
    Let $\epsilon>0$. There exists $\xi>0$ such that the following holds. If $D'$ is an oriented triangle-free graph that has no $C_4$ and satisfies $\delta^+(D') \geq (1/5-\xi)|D'|$, then there exists a sub-digraph $D$ of $D'$ such that $|D|\geq (1-\epsilon) |D'|$ and  for every $u\in V(D)$ and $o\in \{+,-\}$ $$(1/5 - \epsilon)|D|\leq |N^o_{D}(u)| \leq (1/5+\epsilon)|D|.$$
   
\end{cor}
\begin{proof}
Let $\xi \ll \mu \ll \epsilon$. By potentially deleting some edges in $D'$ we may assume that for every $v\in V(D')$, $|N^+_{D'}(v)|= (1/5 - \xi)|D'|$. From Lemma \ref{non-ext-lem01} and Lemma \ref{lem2-non-ext-part1}, there are at most $\mu |D'|$  vertices $v$ such that $|N^-_{D'}(v)|\geq (1/5+\mu)|D'|$. Deleting these vertices from $V(D')$ gives a digraph $D''$ such that for every $u\in V(D'')$, $(1/5-O(\mu))|D''|\leq |N^+_{D''}(u)|\leq (1/5+O(\mu))|D''|$ and $|N^-_{D''}(u)|\leq (1/5+O(\mu))|D''|$.

Let $U:=\{v\in V(D'') : |N^-_{D''}(v)|<(1/5-\sqrt{\mu})|D''|$. We have
$$(1/5-\sqrt{\mu})|D''|\cdot |U| +(1/5+O(\mu))|D''|(|D''|-|U|)> \|D''\| \geq (1/5-O(\mu))|D''|^2.$$
Consequently, $|U|=O(\sqrt{\mu}|D''|)$. Let $D:=D''[V(D'')\setminus U]$. Then $|D|\geq (1-O(\sqrt{\mu}))|D'|$ and for every $u\in D$ and $o\in \{+,-\}$, $(1/5 - O(\sqrt{\mu})|D|\leq |N^o_{D}(u)| \leq (1/5+O(\sqrt{\mu})|D|.$

\end{proof}
%
%
\subsection{Part 2}
Let $\epsilon$ be positive and sufficiently small, let $D$ be the oriented graph from Corollary \ref{cor-NE-part1}, and let $m := \abs{D}$. By Corollary \ref{cor-NE-part1}, for every $u\in V(D)$, \begin{equation}\label{eq-P2-new1}(1/5-\epsilon)m\leq \abs{N^+_D(u)}, \abs{N^-_D(u)}\leq (1/5+\epsilon)m.\end{equation}
We again assume that $D$ has no directed cycle of length four and that $D$ is triangle-free.

Fix ${u_0v_0} \in E(D)$ and let $A_0:=N^-_D(u_0)$, $B_0:=N^+_D(u_0)$, $A_1:=N^-_D(v_0)$, $B_1:=N^+_D(v_0).$ Let $C:=V(D)\setminus (A_0\cup A_1\cup B_0\cup B_1),$ and note that by (\ref{eq-P2-new1}),
\begin{equation}\label{eq0}
(1/5-4\epsilon)m\leq |C|\leq (1/5+4\epsilon)m.
\end{equation}
In addition, all sets $A_0, A_1, B_0, B_1, C$ are pairwise disjoint and $D[A_0\cup B_0]$, $D[A_1\cup B_1]$ have no arcs. Using the fact that $D$ has no $C_4$ we have the following. 
\begin{itemize}
    \item[(a)] For $v\in B_1$, $N^+_D(v)\subseteq B_0\cup C$ and
    \item[(b)] for $w\in N^+_D(B_1)\cap B_0$, $N^+_D(w)\subseteq B_1\cup C$.
\end{itemize}
We will use two auxiliary constants $\eta, \xi$  such that $\epsilon\ll \xi\ll \eta$ and $\eta\rightarrow 0$ when $\epsilon\rightarrow 0$. (For example, $\xi=(7\epsilon)^{1/3}$ and $\eta=\epsilon^{1/64}$ would do.)  
We will now consider two cases.\\

\noindent{\bf Case 1:} $|E_D(B_1, B_0)|\geq \eta m^2$.\\
Let $G$ denote the sub-digraph of $D$ induced by $B_1\cup (N^+_D(B_1)\cap B_0)$. To simplify the notation let $W_1:= B_1$ and $W_0:=N^+_D(B_1)\cap B_0$. 
Rephrasing the assumption of the case, we have
\begin{equation}\label{eq-case1}
    |E_D(B_1, B_0)| = |E_D(W_1, W_0)|=|E_G(W_1, W_0)|\geq \eta m^2.
\end{equation}
In addition, by (a) and (b), for $i=0,1$ every $v\in W_i$, 
\begin{equation}\label{eq1}
    N^+_D(v)\subseteq W_{1-i}\cup C.
\end{equation}
Let $uv\in E(G)$. Since , $N^+_D(u)\cap N^+_D(v)=\emptyset$, in view of (\ref{eq-P2-new1}), (\ref{eq0}) and (\ref{eq1}), 
\begin{equation}\label{eq3}
    |N^+_G(u)|+|N^+_G(v)|\geq 2\cdot(1/5-\epsilon)m - |C|\geq
    (1/5- 6\epsilon)m.
\end{equation}
We start with the following lemma. Let $\widetilde{G}$ be the underlying graph of the digraph $G$, and recall that $\epsilon \ll \xi$.
\begin{lemma}\label{non-ext-lem0}
Let $Z:=  \{w\in W_0\cup W_1: |N^+_G(w)|\leq \xi m\}$. Then at most $2\xi m^2$ pairs $(v,w)\in W_0\times W_1$ are such that $vw\notin E(\widetilde{G})$ and $|\{v,w\}\cap Z|\leq 1$.
\end{lemma}
\begin{proof}
Suppose there are more than $2\xi m^2$ pairs $(v,w)\in W_0\times W_1$ such that $vw \notin E(\widetilde{G})$  and $|\{v,w\}\cap Z|\leq 1.$
Let $Z_i:= W_i\cap Z$, and for $u\in W_i$ let $d^n(u)=|W_{1-i}|- (|N^+_G(u)|+|N^-_G(u)|).$

Let $v\in W_0\setminus Z$ and $w\in W_1\setminus (N^+_G(v)\cup N^-_G(v))$. Then $vw \notin E(\widetilde{G})$ and $|\{v,w\}\cap Z|\leq 1$. Thus
\begin{equation}\label{eq4}
\sum_{w\in (W_0\cup W_1)\setminus Z} d^n(w) \geq 2\xi m^2.
\end{equation}
Let $T:= \{w\in (W_0\cup W_1)\setminus Z: d^n(w)>\xi m\}.$
Then \begin{equation}\label{eq5}|T|\geq {\xi}m,\end{equation} 
since otherwise, $$\sum_{w\in (W_0\cup W_1)\setminus Z} d^n(w) < |T|\cdot m+ {\xi}m(m-|T|)< 2\xi m^2 $$
contradicting (\ref{eq4}).
We have $|W_i|\leq (1/5+\epsilon)m$, and so, for every $w\in T$, $$|N^+_G(w)|+|N^-_G(w)|< (1/5+\epsilon -\xi)m.$$
Let $q:=\sum_{w\in W_0\cup W_1} |N^+_G(w)|(|N^+_G(w)|+|N^-_G(w)|).$ Then
$$q \leq(1/5+\epsilon)m\sum_{w\in Z}|N^+_G(w)| + \sum_{w\in (W_0\cup W_1)\setminus Z} |N^+_G(w)|(|N^+_G(w)|+|N^-_G(w)|)< $$
$$(1/5+\epsilon)m\cdot ||G|| - \xi m\sum_{w\in T}|N^+_G(w)|<(1/5+\epsilon -\xi^3)m\cdot ||G||,$$
because by (\ref{eq5}), $|T|\geq \xi m$, and if $w\in T$, then $w\notin Z$ and so $|N^+_G(w)|>\xi m.$
At the same time, by (\ref{eq3}),
$$q= \sum_{uv\in E(G)}(|N^+_G(u)|+|N^+_G(v)|)\geq (1/5-6\epsilon)m\cdot ||G||$$
contradicting the definition of $\xi$.
\end{proof}
We will now show that $\|G\|$ is approximately $|W_0|\cdot |W_1|.$ 
\begin{lemma}\label{non-ext-lem1}
 $||G||\geq |W_0|\cdot |W_1|-\sqrt{\xi}m^2.$
\end{lemma}
\begin{proof}Suppose for at least $\sqrt{\xi} m^2$ pairs $(v,w)\in W_0\times W_1$, $vw \notin E(\widetilde{G})$. 

Let $Z$ be as in Lemma \ref{non-ext-lem0} and let $Z_i:= W_i\cap Z$. By Lemma \ref{non-ext-lem0} we may assume that all but at most $2\xi m^2$ pairs $(v,w)\in W_0 \times W_1$ are such that if $\{v,w\}\notin E(\widetilde{G})$ then $\{v,w\}\subseteq Z.$\\
First note that we may assume   \begin{equation}\label{eq-ne-part2-n1}|Z_0|>\sqrt{\xi}m.
\end{equation}Indeed, if  $|Z_0|\leq \sqrt{\xi}m$, then the number of pairs $(v,w)\in W_0 \times W_1$ such that  $\{v,w\}\notin E(\widetilde{G})$ is less than $2\xi m^2+ |Z_0|\cdot |Z_1|< \sqrt{\xi}m^2$, and so the lemma holds.

By the definition of $Z$ for every $w\in Z$, $|N^+_G(w)|\leq \xi m$. Therefore, if $w\in Z$ and $v\in (W_0\cup W_1)\setminus Z$ are such that $vw \in E(\widetilde{G})$ (i.e. $wv\in E(G)$ or $vw\in E(G)$), then, by (\ref{eq3}), 
\begin{equation}\label{eq-ne-part2-n2}|N^+_G(v)|\geq (1/5-6\epsilon - \xi)m.
\end{equation}

Let $T_1$ be the set of vertices $v\in W_1\setminus Z_1$ such that for every $w\in Z_0$, $vw \notin E(\widetilde{G}).$
Then, by Lemma \ref{non-ext-lem0}, $|T_1|\cdot |Z_{0}| \leq 2\xi m^2$, and so by (\ref{eq-ne-part2-n1}),
$$|T_1|\leq 2\xi m^2/|Z_{0}|<2\sqrt{\xi} m.$$

\begin{claim}\label{nonext-case2-cl1n}
    $|W_0|\geq (1/5-2\xi)m$
\end{claim}
{\it{Proof of Claim \ref{nonext-case2-cl1n}}.}
By the definition of the case,  $|E_D(W_1, W_0)|=|E_D(B_1, N^+_D(B_1)\cap B_0)|=|E_D(B_1, B_0)|\geq \eta m^2.$ Consequently, $|W_1\setminus (Z_1\cup T_1)|\geq \eta m/2$, since otherwise
$$|E_D(W_1, W_0)|<|W_1\setminus (Z_1\cup T_1)|m+ |Z_1|\xi m + |T_1|m<\eta m^2.$$
By (\ref{eq-ne-part2-n2}), for $v\in W_1\setminus (Z_1\cup T_1)$, $|N^+_G(v)|\geq (1/5-2\xi)m$. Therefore, since $W_1\setminus (Z_1\cup T_1)\neq \emptyset$, by the definition of $W_0$, $|W_0|\geq  (1/5-2\xi)m$. 
$\Box$

From Claim \ref{nonext-case2-cl1n}, we have
\begin{equation}\label{eq-ne-part2-n3}|E_D(W_0\cup W_1, V(D))|\geq (1/5-\epsilon)m (|W_0|+|W_1|)\geq (2/25-3\xi)m^2, \end{equation}
and at the same time, by (a) and (b),
\begin{equation}\label{eq-ne-part2-n4}|E_D(W_0\cup W_1, V(D))|\leq ||G||+|E_D(W_0\cup W_1, C)|.
\end{equation}
By (\ref{eq-P2-new1}) and (\ref{eq0}), $|E_D(W_0\cup W_1, C)|\leq |C|(1/5+\epsilon)m \leq (1/25 +\xi)m^2.$
 Consequently, by (\ref{eq-ne-part2-n3}) and (\ref{eq-ne-part2-n4}),
$$||G||\geq (1/25- 4\xi)m^2> |W_0|\cdot |W_1| -\sqrt{\xi}m^2.$$

\end{proof}
We will now show a stronger statement. We argue that $W_0$ is approximately the whole $B_0$ and that almost all edges in $D$ between $B_1$ and $B_0$ are directed from $B_1$ to $B_0.$
\begin{lemma}\label{non-ext-lem2}
$|E_D(B_1, B_0)|\geq (1-\mu)|B_0|\cdot |B_1|$ where $\mu=O(\xi^{1/8}/\sqrt{\eta})$. 
\end{lemma}
\begin{proof}
In view of Lemma \ref{non-ext-lem1},
$$|E_G(W_1, W_0)|+|E_G(W_0,W_1)|\geq |W_0|\cdot |W_1|- \sqrt{\xi} m^2.$$
Let $\nu:=\frac{\xi^{1/4}}{\eta}$ and note that $\nu \ll \eta$. Let $W_i'= \{v\in W_i: |W_{1-i} \cap (N^+_G(v)\cup N^-_G(v))| \geq (1-\nu)|W_{1-i}|\}.$
We have
$$|W_1|\cdot |W_0| -\sqrt{\xi} m^2\leq |W_i'|\cdot |W_{1-i}|+ (1-\nu)|W_i\setminus W_i'|\cdot |W_{1-i}|\leq |W_1|\cdot |W_0|-\nu |W_i\setminus W_i'|\cdot |W_{1-i}|.$$
Therefore, \[
    \sqrt{\xi}m^2 \geq \nu|W_i \setminus W'_i|\cdot |W_{1 - i}|
\]
Since $|E_D(W_1,W_0)|=|E_D(B_1, B_0)|\geq \eta m^2$, $\min{|W_{1-i}|, |W_i|}\geq \eta m$, and so using the definition of $\nu$,
\begin{equation}\label{eq9}
|W_i\setminus W_i'|\leq \frac{\sqrt{\xi}m^2}{\nu |W_{1-i}|}\leq \xi^{1/4}m\leq \nu |W_i|.\end{equation}
Consequently,
$$|W_i'|\geq (1-\nu)|W_i|.$$

Let $t:= \min_{v\in W_0'}|N^+_{G}(v)|$ and suppose $x\in W_0'$  satisfies $t= |N^+_{G}(x)|.$ 
Then 
\begin{equation}\label{eq-ne-t}
    t\leq (1/5-\eta)m.
\end{equation}
Indeed, if $t>(1/5-\eta)m$, then using the fact that $G$ is an oriented graph, $|W_i|< m/4$ and (\ref{eq9}), we have

$$|E_D(W_1, W_0)|\leq |W_1|\cdot|W_0\setminus W_0'|+|W_1\setminus W_1'|\cdot |W_0'|+ (\eta+\epsilon)m|W_1'|< (\eta/2 +\nu)m^2,$$
which is less than $\eta m^2$ since $\nu \ll \eta$. 

By the definition of $x$, from (\ref{eq3}), for every $z\in N^+_{G}(x)\cup N^-_{G}(x)$,  $|N^+_{G}(z)|\geq (1/5-6\epsilon)m- t.$ 
Consequently, since $x\in W_0'$, $|N^+_{G}(x)\cup N^-_{G}(x)|\geq (1-\nu)|W_1|$, and so
\begin{equation}\label{eq-E_G}
    |E_{G}(W_1, W_0)|\geq (m/5-t -6\epsilon m)(1-\nu)|W_1|\geq (m/5- t-2\nu m)|W_1|.
\end{equation}
We now consider two cases based on the value of $t$.

If $t\leq \sqrt{\nu}m$, then from (\ref{eq-E_G}), 
$$|E_D(B_1, B_0)|\geq (m/5- 3\sqrt{\nu} m)|W_1|\geq (1- O(\sqrt{\nu}))|B_1|\cdot |B_0|,$$
because $|B_1|=|W_1|$ and $|B_0|\leq m/5+\epsilon m$.

Now suppose $t>\sqrt{\nu}m$. We will show that this case leads to a contradiction. Since $|E_G(W_1, W_0\setminus W_0')|<\nu |W_1|\cdot |W_0|$, from (\ref{eq-E_G}), there exists $v\in W_1'$ such that $|N_G^-(v)|\geq m/5 -t -3\nu m$. Let $A_1= N^-_G(v)$, $A_2=N^+_G(v)$, and $A_3=N^+_G(B)$. Then $A_1\cap A_2=\emptyset$, $|A_2|\geq t $ and since there is no $C_4$, for every $w\in A_3$, $N^+_G(w)\cap A_1=\emptyset$. Since $|W_0\setminus A_1|<t+4\nu m$, we have
$$(t+4\nu m)|A_3|> |E_G(W_0\setminus A_1, A_3)|+ |E_G(A_3, W_0\setminus A_1)|\geq |E_G(A_2, A_3)|+t|A_3|.$$
 Thus
 $$4\nu m |A_3|> |E_G(A_2,A_3)|\geq (m/5-t-6\epsilon m)|A_2|\geq (m/5-t-6\epsilon m)t.$$
However, since $\nu \ll \eta$ and $t>\sqrt{\nu}m$, by (\ref{eq-ne-t}),
$(m/5-t-6\epsilon m)t> \sqrt{\nu} m^2/6$, contradicting $\nu \ll \eta \leq 1$.

\end{proof}

 We have the following simple observation.
\begin{lemma}\label{non-ext-lem-comput}
Let $0<\xi_1, \xi_2<1/2$. For three distinct non-empty sets $X, Y,Z$ of vertices if $|E_D(X,Y)|\geq (1-\xi_1)|X|\cdot |Y|$ and $|E_D(Y,Z)|\geq (1-\xi_2)|Y|\cdot |Z|$, then
$$|E_D(Z, X)|\leq 2(\xi_1+\xi_2)|Z|\cdot |X|.$$
\end{lemma}
\begin{proof}
For at least $(1-2\xi_2)|Z|$ vertices $v\in Z$, $|N^-_D(v)\cap Y|\geq |Y|/2$. If for any such $v$, $|N^+_D(v)\cap X|>2\xi_1|X|$, then there is a triangle in the underlying graph. Thus $|E_D(Z, X)|\leq 2(\xi_1+\xi_2)|Z|\cdot |X|.$ 
\end{proof}
We have the following lemma which in addition to Case 1, will be used in Case 2.
\begin{figure}[h] 
    \centering
    \includegraphics[width=1\textwidth]{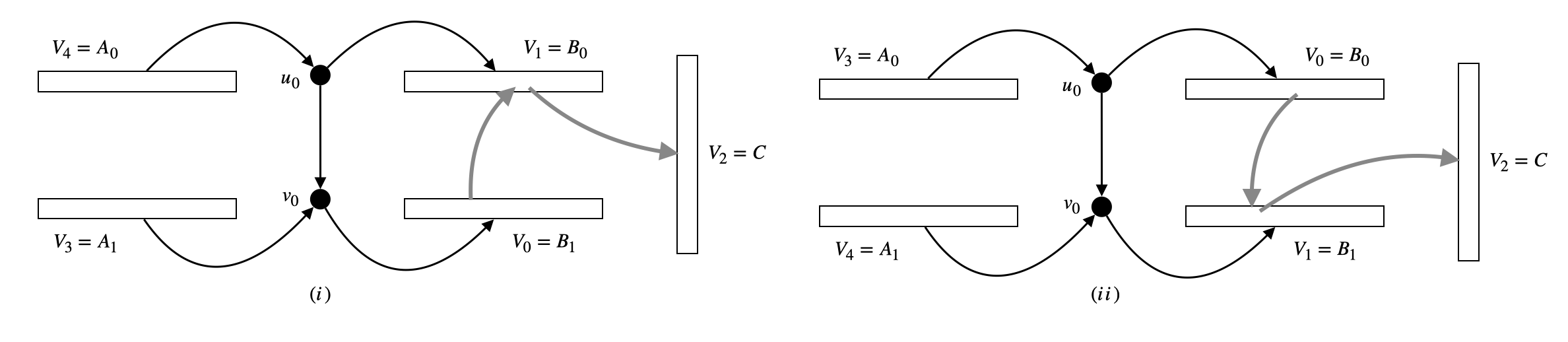} 
    \caption{Two cases in Lemma \ref{non-ext-lem3}}
    \label{fig:Lem-3-10}
\end{figure}
\begin{lemma}\label{non-ext-lem3}
Let $0<\epsilon \ll \xi$ and suppose $|E_D(V_0, V_1)|\geq (1-\xi)|V_0|\cdot |V_1|$ and $|E_D(V_1, V_2)|\geq (1-\xi)|V_1|\cdot |V_2|$ where $\{V_0,V_1\}=\{B_0, B_1\}$ and $V_2= C$. Let $(V_3, V_4)= (A_1, A_0)$ if $(V_0, V_1)=(B_1, B_0)$ and $(V_3, V_4)= (A_0, A_1)$ if $(V_0, V_1)=(B_0, B_1)$. Then
$V_0, V_1, V_2, V_3, V_4$ contains a $(m/5, \mu)$-blow up of a directed cycle of length five for $\mu=O(\xi).$
\end{lemma}
\begin{proof}
From (\ref{eq0}), for every $i$, $$(1-20\epsilon)m/5\leq |V_i|\leq (1+20\epsilon)m/5,$$
and  by the assumption $\epsilon \ll \xi$.

\begin{claim}\label{cl1-lem3} $|E_D(V_2, V_3)|\geq (1-9\xi)|V_2|\cdot |V_3|.$
\end{claim}
\noindent{\it Proof of Claim \ref{cl1-lem3}.} We have \begin{equation}\label{eq1-cl1-lem3}|E_D(V_2, V_1)|\leq \xi |V_1|\cdot |V_2|.\end{equation} 

In addition, by (\ref{eq-P2-new1}), \begin{equation}\label{eq2-cl1-lem3}|E_D(V_2, V_2)|\leq (1/5+\epsilon)m|V_2| - (1-\xi) |V_1|\cdot |V_2|\leq 2\xi |V_2|^2.\end{equation}
From Lemma \ref{non-ext-lem-comput}, we also have
\begin{equation}\label{eq3-cl1-lem3}
|E_D(V_2, V_0)|\leq 4\xi |V_0|\cdot |V_2|.
\end{equation}
Since $|E_D(V_1, V_2)|\geq (1-\xi)|V_1|\cdot |V_2|$, for at least $(1-\xi) |V_2|$ vertices $v\in V_2$, $|N^-_D(v)\cap V_1|>0$ and for any such vertex $v\in V_2$, $N^+_D(v)\cap V_4=\emptyset$ or there is a $C_4$. Thus
\begin{equation}\label{eq4-cl1-lem3}
|E_D(V_2, V_4)|\leq\xi |V_2|\cdot |V_4|.
\end{equation}
Consequently, by (\ref{eq1-cl1-lem3}), (\ref{eq2-cl1-lem3}), (\ref{eq3-cl1-lem3}), and (\ref{eq4-cl1-lem3}), $$|E_D(V_2, V_3)|\geq (1/5-\epsilon)m|V_2|- 8\xi (1/5+\epsilon)m|V_2| >(1-9\xi)|V_2|\cdot |V_3|.$$
$\Box$

\begin{claim}\label{cl2-lem3} $|E_D(V_3, V_4)|\geq (1-O(\xi))|V_3|\cdot |V_4|.$
\end{claim}
\noindent{\it Proof of Claim \ref{cl2-lem3}.} Since $D$ is triangle-free, \begin{equation}\label{eq1-cl2-lem3}
|E_D(V_3, V_0)|=0.
\end{equation}
From Claim \ref{cl1-lem3},
\begin{equation}\label{eq2-cl2-lem3}
|E_D(V_3, V_2)|\leq 9\xi |V_3|\cdot |V_2|
\end{equation}
repeating the computations from (\ref{eq2-cl1-lem3}) and using (\ref{eq-P2-new1}) and Claim \ref{cl1-lem3} again, we have
\begin{equation}\label{eq3-cl2-lem3}
|E_D(V_3, V_3)|\leq O(\xi |V_3|^2).
\end{equation}
In addition, by Lemma \ref{non-ext-lem-comput}, we have
\begin{equation}\label{eq4-cl2-lem3}
|E_D(V_3, V_1)|\leq O(\xi |V_1|\cdot |V_3|).
\end{equation}
Thus, by (\ref{eq1-cl2-lem3}), (\ref{eq2-cl2-lem3}), (\ref{eq3-cl2-lem3}), and (\ref{eq4-cl2-lem3}), $|E_D(V_3, V_4)|\geq (1-O(\xi))|V_3|\cdot |V_4|$. $\Box$

Using a similar argument to the proof of Claim \ref{cl2-lem3}, we also have $|E_D(V_4, V_0)|\geq (1-O(\xi))|V_4|\cdot |V_0|.$
\end{proof}
 \begin{cor}\label{non-ext-cor1}
 $B_1,B_0,C, A_1, A_0$ contains a $(m/5,\mu)$-blow up of a directed cycle of length five for $\mu=O(\mu_{Lem \ref{non-ext-lem2}})$.
 \end{cor}
 \begin{proof}
Let $V_0:=B_1, V_1:=B_0, V_2:=C, V_3:=A_1, V_4:=A_0$, (see Figure \ref{fig:Lem-3-10}$(i)$)
 and let $\xi = \mu_{Lem. \ref{non-ext-lem2}}.$ By Lemma \ref{non-ext-lem2}, \begin{equation}\label{eq1-lem3}|E_D(V_0, V_1)|\geq (1-\xi)|V_0|\cdot |V_1|.
\end{equation}
Consequently, $|N^+_D(V_0)\cap V_1|\geq (1-\xi)|V_1|$ and for every $v\in N^+_D(V_0)\cap V_1$, $N^+_D(v)\subseteq V_0\cup V_2$. Thus

\begin{equation}\label{eq2-lem3}|E_D(V_1, V_2)|\geq (1/5-\epsilon)m |V_1|- O(\xi m^2)\geq (1-O(\xi))|V_1|\cdot |V_2|.\end{equation}
By Lemma \ref{non-ext-lem3}, $V_0, V_1, V_2, V_3, V_4$ contains a $(m/5, \mu)$-blow up of a directed cycle of length five for $\mu=O(\xi).$
 \end{proof}
We will now proceed to the second case which is substantially shorter.\\
 \noindent {\bf Case 2:} $|E_D(B_1, B_0)|\leq \eta m^2.$\\
Recall that $\epsilon \ll \eta \ll 1$, and that for every $v
 \in B_1$, $N^+_D(v)\subseteq B_0\cup C$. Consequently,
 \begin{equation}\label{eq1-case2}
 |E_D(B_1, C)|\geq |B_1|(1/5-\epsilon)m- \eta m^2,
 \end{equation}
 and so, 
\begin{equation}\label{eq1-case2.1}|E_D(B_1, C)|\geq (1-O(\eta))|B_1|\cdot |C|.\end{equation}
 From (\ref{eq1-case2}) and (\ref{eq-P2-new1}), for every $X\in \{A_0,A_1, B_0, C\}$
 \begin{equation}\label{eq2-case2}
 |E_D(X, C)|<2\eta m^2,
 \end{equation}
 as otherwise, $|E_D(X\cup B_1, C)|>m^2/25+\eta m^2/2$, and as a result
 there is $v\in C$ such that $|N^-_D(v)|>(1/5+\epsilon)m.$
 \begin{cor}\label{non-ext-cor2}
 $B_0, B_1,C, A_0, A_1$ contains a $(m/5,\mu)$-blow up of directed cycle of length five for $\mu=O(\sqrt{\eta})$.
 \end{cor}
 \begin{proof}
 Let $V_0:=B_0, V_1:=B_1, V_2:=C, V_3:=A_0, V_4:=A_1$  (see Figure \ref{fig:Lem-3-10}$(ii)$). 
We claim that
 \begin{equation}\label{eq3-case2}
 |E_D(V_0, V_4)|<\sqrt{\eta}|V_0|\cdot |V_4|.
 \end{equation}
 Indeed, suppose (\ref{eq3-case2}) does not hold and we have $ |E_D(V_0, V_4)|\geq \sqrt{\eta}|V_0|\cdot |V_4|.$ Let $W:= V_4\cap N^+_D(V_0).$ Since $|E_D(V_0,V_4)|\leq |V_0|\cdot |W|$,  $|W|\geq \sqrt{\eta}|V_4|$. For every $v\in W$, $N^+_D(v)\cap V_3=\emptyset$ or there is a $C_4$. In addition, $E_D(W, V_4\cup V_1)=\emptyset$. Thus, by (\ref{eq2-case2}), 
 $$|E_D(W, V_0)|> |W|(1/5-\epsilon)m- 2\eta m^2>(1-\sqrt{\eta})|W|\cdot |V_0|,$$
 contradicting $|E_D(V_0, W)|=|E_D(V_0, V_4)|\geq \sqrt{\eta}|V_4|\cdot |V_0|\geq \sqrt{\eta}|W|\cdot |V_0|.$ 

 By (\ref{eq3-case2}) and (\ref{eq2-case2}),
 \begin{equation}
 |E_D(V_0, V_1)|\geq (1/5-\epsilon)m|V_0|- O(\sqrt{\eta}m^2)\geq (1-O(\sqrt{\eta}))|V_0|\cdot |V_1|,
 \end{equation}
 and by (\ref{eq1-case2.1}), $|E_D(V_1, V_2)|\geq (1-O(\eta))|V_1|\cdot |V_2|.$ Therefore, by Lemma \ref{non-ext-lem3}, $V_0, V_1, V_2, V_3, V_4$ contains a $(m/5, \mu)$-blow up of a directed cycle of length five for $\mu=O(\sqrt{\eta})$.
 \end{proof}
 \section{The Extremal Case}\label{ext-section}
Let $H$ be an edge-colored triangle-free graph on  $n$ that does not contain a monochromatic path of length three. Let $D_H$ be the associated digraph. Then $\delta^+(D_H)= \delta^c(H)>n/5$. Given $0<\eta \ll 1$, we choose $0<\epsilon\ll d \ll \eta$ and $t\in \mathbb{Z}^+$ and we apply the regularity lemma to $D_H$ to obtain the reduced digraph $\mathcal{R}_D:=\mathcal{R}_D(\epsilon, d, t)$ on $m\geq t$ vertices. By (\ref{deltsR}),  $\delta^+(\mathcal{R}_D)\geq (1/5-O(d))m$. Now, using the results from the previous section either $\mathcal{R}_D$ contains a $C_4$ or it contains an $(m/5, \mu)$-blow up of a directed cycle of length five for some $\mu \ll \eta$, and we have the following fact.  

 \begin{fact}\label{fact-about-C4}
 Let $\eta>0$. There exist $0<\epsilon\ll d\ll \eta$ and $l\in \mathbb{Z}^+$ such that either $\mathcal{R}_D:=\mathcal{R}_D(\epsilon, d, l)$ contains a $C_4$ or $D_H$ is $\eta$-extremal. 
 \end{fact}
 \begin{proof}
     We will only outline the proof. Given $\eta>0$ let $\mu$ be such that the conclusion of Lemma \ref{blow-up-extremal-lem} holds. Further set $\epsilon \ll d\ll \mu$ and suppose the reduced digraph of $D_H$, $\mathcal{R}_D$, has no $C_4$. Let $m=|\mathcal{R}_D|$. Then, in view of Corollary \ref{non-ext-cor1} and Corollary \ref{non-ext-cor2}, $\mathcal{R}_D$ contains an $(m/5,\mu)$-blow-up of a directed cycle of length five. Consequently, by Lemma \ref{blow-up-extremal-lem}, $D_H$ is $\eta$-extremal. 
 \end{proof}
 In this section, we will address the case when $D_H$ is $\eta$-extremal for a sufficiently small $\eta>0$. We always assume that the order of $D_H$ is sufficiently large and we will use two auxiliary constants, $\gamma>0$ and $\beta>0$ that satisfy
 \begin{equation}\label{ext-eta-beta-gamma}
\eta \ll \beta \ll\gamma  \ll 1.
\end{equation}
The rest of this section is organized as follows. In Section \ref{ext-subsec-1}, we establish a few useful properties of an extremal digraph. In Section \ref{ext-subsec-2}, we prove main lemmas for the extremal case and prove the Theorem \ref{rainbowgeneral-thm}. Finally, in Section \ref{ext-subsec-3}, we analyze length four and prove Theorem  \ref{rainbowC4-thm}.
\subsection{Initial Observations}\label{ext-subsec-1}
Suppose that $D_H$ is $\eta$-extremal, let $V_0,\dots, V_4$ be the sets from Definition \ref{ext-def}, and let $c:E(H)\rightarrow \mathcal{C}$ denote the edge coloring of $H$.
\begin{defn}
For $i=0,\dots, 4$, let $V_i'$ be the set of vertices $v\in V_i$ such that for some color $c_v$  all but at most $\beta |V_{i-1}|$ vertices $w\in N^-_{D_H}(v)\cap V_{i-1}$ satisfy $c(wv)=c_v$.
\end{defn}
We first show that given $k\in \mathbb{Z}^+$ and assuming $k\ll |D_H|$, either there is a rainbow cycle of length $4k$ or for every $i$ almost all vertices from $V_i$ are in $V_i'$. 

\begin{lemma} \label{lem-V'}  If $|V_i'| < (1-\beta)|V_i|$, there is a rainbow cycle  of length $4k$ in $H$.
\end{lemma}

\begin{proof}Recall that by the definition of $D_H$ for every $v$ all edges in $E(D_H)$ from $v$ to the out-neighbors of $v$  have different colors and that $\eta \ll \beta \ll 1$. For every $v\in V_i\setminus V_i'$ there are two subsets $Z_1(v), Z_2(v)$ of $N_{D_H}^-(v)\cap V_{i-1}$ such that $Z_1(v)\cap Z_2(v)=\emptyset$, $|Z_1(v)|\geq \beta |V_{i-1}|$, $|Z_2(v)|\geq \beta |V_{i-1}|$, and the colors on the edges between $Z_1(v)$ and $v$ are different than the colors on the edges between $Z_2(v)$ and $v$. This gives an auxiliary 2-coloring of a subgraph of $H[V_{i-1}, V_i\setminus V_i']$ with $wv$ colored $j\in \{1,2\}$ if $w\in Z_j(v).$
Suppose $|V_i\setminus V_i'|\geq \beta |V_i|.$ Then  there are sets $S\subseteq V_i\setminus V_i'$ and $T_1, T_2\subseteq V_{i-1}$ such that $|S|=10k$, $|T_1|=|T_2|=k$ and in the auxiliary 2-coloring, $T_j\cup S$ induces the complete bipartite graph of color $j\in \{1,2\}$.
We will now construct a rainbow $C_{4k}$ connecting $k$ vertices from each of the $T_1, T_2$ one by one  using the total of $2k$ additional vertices from $S$. Recall that for every $w\in T_j$ all edges between $w$ and $S$ have distinct colors in $H$. 
Let $w_1\in T_1$ and $w_2\in T_2$. It is enough to argue that there is a rainbow path $w_1vw_2$  for some $v\in S$ that avoids any $4k-2$ colors and any $2k-1$ vertices $u\in S$. There can be at most $8k-4$ vertices $u$ in $S$ such that $c(w_1u)$ or $c(w_2u)$ is among $4k-2$ forbidden colors. Consequently, there is $v\in S$ that is not among $2k-1$ forbidden vertices, and such that $c(w_1v), c(w_2v)$ are not forbidden. Since $w_1\in T_1$ and $w_2\in T_2$, $c(w_1v)\neq c(w_2v)$.
\end{proof} 
\begin{defn}
     For $i=0, \dots, 4$ and $v\in V_i'$, let $N^m(v)$ be the set of vertices $x\in N^-_{D_H}(v)\cap V_{i-1}'$ such that $c(xv)=c_v$.
 \end{defn}
 For $v\in V_i'$,
 \begin{equation}\label{ext-Nm}
 |N^m(v)|\geq |V_{i-1}'|- \eta |V_{i-1}|-\beta|V_{i-1}|\geq (1-3\beta/2)|V_{i-1}'|.
 \end{equation}
The following simple observation will be used repeatedly when dealing with vertices in $V_i'$.
 \begin{fact}\label{ext-fact3}
 For $i=0, \dots, 4$, if $u, v \in V_i'$ are such that $u\neq v$, then $c_u\neq c_v$. 
 \end{fact}
 \begin{proof}
 For $x\in \{u,v\}$, from (\ref{ext-Nm}),
 $|N^m(x)|\geq (1-3\beta/2)|V_{i-1}'| >|V_{i-1}'|/2.$ Consequently $N^m(u)\cap N^m(v)\neq \emptyset$. However, if $z\in N^m(u)\cap N^m(v)$, then by the definition of $D_H$, $c(zu)\neq c(zv)$.
 \end{proof}
{Before proceeding, we slightly modify  the digraph $D_H$. Note that for every $v\in \bigcup_{i=0}^4 V_i'$ there is exactly one vertex $w_v\in V(H)$ such that $vw_v\in E(D_H)$ and   $c(vw_v)=c_v$.

\begin{defn}\label{def-ext-D}Let $D$ be the digraph obtained from $D_H$ as follows.  If $w_v\notin N^m(v)$ then choose arbitrarily one vertex $u_v\in N^m(v)$, add $vu_v$ to $D$, and delete the edge  from $v$ to $w_v$.\end{defn}
Since the out-degree of any vertex is unchanged, for every $v\in V(H)$, $d_D^+(v)= d_{D_H}^+(v)= d^c(v)$. Also, from Definition \ref{ext-def}, for $v\in  V_i'$, 
\begin{equation}\label{eq-new}|N^+_D(v)\cap V_{i+1}|\geq (1-\eta)|V_{i+1}|-1\end{equation}
and trivially
$$|N^-_D(v)\cap V_{i-1}|\geq |N^m(v)|.$$
In addition, for $v\in \bigcup_{i=0}^4 V_i'$, if $vw\in E(D)$ and $w\notin N^m(v)$, then $c(vw)\neq c_v$.}

 \begin{defn}\label{ext-def-V''}
 Let $V_0'', \dots, V_4''$ be obtained as follows. For every $v\in V(H)\setminus \bigcup_{i=0}^4 V_i'$, add $v$ to $V_i''$ if $|N_D^+(v)\cap V_{i+1}'|\geq \gamma|V_{i+1}'|$  and if possible $|N_D^+(v)\cap V_{i-1}'|\geq \gamma |V_{i-1}'|$. 
 \end{defn}
Note that we do not necessarily have that $V_i'' \subseteq V_i$.  Since $\delta^+(D)\geq n/5$ and $\gamma \ll 1$, $\{V_0'\cup V_0'', \dots, V_4'\cup V_4''\}$ is a partition of $V(H).$ In addition, from Lemma \ref{lem-V'},
 \begin{equation}\label{eq-V_i''}
     |V_i''|\leq (\beta+\eta) n\leq 6\beta |V_i'|.
 \end{equation}

 We have the following observation. In this fact and the rest of the section we always assume that the addition is modulo five.
 \begin{fact}\label{ext-fact2}
 \begin{itemize}
\item[(a)]     For every $v\in V_i'\cup V_{i}''$, $N_H(v)\cap( V_{i}'\cup V_{i+2}')=\emptyset$. 
\item[(b)] For every $i$ and every $v\in V_i'$, $N_H(v)\subseteq (V_{i-1}'\cup V_{i-1}'') \cup (V_{i+1}'\cup V_{i+1}'')\cup V_{i+2}''.$
\end{itemize}
 \end{fact}
 \begin{proof}
Let $v\in V_i'\cup V_i''$. Then $|N_D^+(v)\cap V_{i+1}'|\geq \gamma |V_{i+1}'|$. If $w\in V_i'\cup V_{i+2}'$ then by (\ref{ext-Nm}), $w$ has at least $(1-2\beta)|V_{i+1}'|$ neighbors in $V_{i+1}'$. Since $\beta\ll \gamma $, $w$ and $v$ have a common neighbor in $V_{i+1}'$, which proves (a). If $v\in V_i'$ and $w\in V_i''$, then $v$ and $w$ share a neighbor in $V_{i+1}'$, and similarly if $w\in V_{i+3}'\cup V_{i+3}''$, then $v$ and $w$ share a neighbor in $V_{i-1}'$. 
 \end{proof}
 In particular, from Fact \ref{ext-fact2}(a), a vertex can be added to sets $V_i''$ only for non-consecutive indices $i$. We continue with a lemma which shows that in many cases we can obtain a rainbow even cycle of a fixed length.
 
 \begin{lemma}\label{ext-lem2}
    If there exists $i$ such that for some $u\in V_{i-1}'$ and $v\in V_i'$, $uv\in E(H)$ and $c_u, c_v, c(uv)$ are three distinct colors, then $H$ contains a rainbow  $C_{2k}$.
 \end{lemma}
 \begin{proof}
 Let $U:= \{c_u, c_v, c(uv)\}$. Let $z_1\in N^m(u)$ be such that $c_{z_1}\notin U$, and set $U:= U \cup \{c_{z_1}\}$. Select distinct vertices $z_2, \dots, z_{k-1}$ one by one from $V(H)\setminus \{u, v, z_1\}$ and modify $U$ as follows. For every $j$, select $z_{j+1}\in N^m(z_j)$ so that $c_{z_{j+1}}\notin U$ and add $c_{z_{j+1}}$ to $U$. Now select $y_{k-2}, \dots, y_1$ in a similar way as follows. Supposing  $y_{k-1}:=z_{k-1}$, for every $j=1, \dots, k-2$ if $y_{j+1}\in V_{l}'$ then select
a vertex $y_{j}\in N^+_D(y_{j+1})\cap V_{l+1}'$ that has not been selected so far, so that $c(y_{j+1}y_j)\notin U$ and set $U:=U\cup \{c(y_{j+1}y_j)\}$. Finally, select $w\in V_{i-1}'$ that has not been selected so far so that $w\in N^+_D(y_{1})\cap N^m(v)$ and $c(y_1w)\notin U.$
 \end{proof}
 In addition to  Lemma \ref{ext-lem2} it is not possible to have $c_u=c_v=c(uv)$ because this leads to a monochromatic path of length three. However, the option of two of these colors being equal cannot be easily excluded. 

 \begin{defn}
 An edge $vw$ in $D$ is called special if for some $i$, $v\in V_i'$ and $w\in V_{i-1}'\cup V_{i-1}'',$ and $c(vw) \neq c_v$. A vertex $w$ in $V_{i-1}'\cup V_{i-1}''$ is called special if it is incident to at least one special edge $vw\in E(D)$. 
 \end{defn}
 Note that if $vw_1$ and $vw_2$ are two special edges with $v\in V_i'$ and $w_1, w_2\in V_{i-1}'\cup V_{i-1}''$, then $c(vw_1)\neq c(vw_2)$ because both $w_1$ and $w_2$ are in the out-neighborhood of $v$. 

 \begin{defn}
For $x\in V_i'$ and $y\in V_j'$, a rainbow $x, y$-connector is a path $xzy$ in $H$ such that $c_x, c_y, c(xz), c(zy)$ are four different colors.
\end{defn}
Rainbow connectors can be used to get rainbow cycles of an even length.
 \begin{lemma}\label{ext-lem-path}
Let $k\geq 2$. Suppose for some $i$ and $j\in \{i, i-2\}$ there is a rainbow $x,y$-connector for some $x\in V_{i}'$ and $y\in V_j'$.  Then $H$ contains a rainbow $C_{2k}.$ 
\end{lemma}
\begin{proof}
Suppose $xzy$ is a rainbow $x,y$-connector.
First, consider the case $j=i$. For $C_4$ let $w\in (N^m(x)\cap N^m(y))\setminus \{z\}.$ Then $xzywx$ is a rainbow $C_4$. For $k>2$, let $U:= \{c_y, c_x, c(xz), c(yz)\}$ and $W:=\{x,y,z\}$. Select $x_0\in N^m(x)\setminus W$ so that $c_{x_0}\notin U$ and add $x_0$ to $W$, $c_{x_0}$ to $U$. Select additional $k-3$ vertices $x_1, \dots, x_{k-3}$ one by one so that for every $j$, $x_j\in N^m(x_{j-1})\setminus W$ and $c_{x_j}\notin U$; add $c_{x_j}$ to $U$ and $x_j$ to $W$. Select $z'\in N^m(x_{k-3})\setminus W$, add $z'$ to $W$, and 
proceed backward to close the cycle. 

The proof in the case of $j=i-2$ is very similar. For $k=2$, let $w\in N^m(x)\cap N^+_D(y)$ be such that $c(yw)\notin \{c_x, c(xz), c(yz)\}$. For $k\geq 3$, let $w\in N^m(x)$ be such that $c_w\notin \{c_x, c_y, c(xz), c(yz)\}$ and let $v\in N^m(w)$ be such that $c_v\notin \{c_x, c_y, c_w, c(xz), c(yz)\}$. Then $y,v \in V_j'$ and $yzxwv$ is a rainbow path of length four that does not contain $c_y, c_v$. Thus we can proceed as in the previous case.
\end{proof}

The following observation will be used to obtain rainbow connectors.

 \begin{fact}\label{ext-fact4.5}
 \begin{itemize}
 \item[(a)] 
 If for some $z\in V(H)$ and some $i\in \{0,\dots, 4\}$ there exist four vertices $x_1, x_2, x_3, x_4\in V_i'$ such that colors $c(x_1z), c(x_2z), c(x_3z), c(x_3z)$ are different, and for every $j\in \{1, \dots, 4\}$, $c(x_jz)\neq c_{x_j}$, then there is a rainbow $x,y$-connector for some $x,y\in V_i'$.
\item[(b)] If $z\in V_{i-1}'\cup V_{i-1}''$ is incident to at least four special edges $vz$ with $v\in V_i'$, then there is a color $\alpha_z$ such that all but at most one of these special edges have color $\alpha_z$ or there is a rainbow $x,y$-connector for some $x, y\in V_i'$.
\end{itemize}
 \end{fact}
 \begin{proof}
 For (a),  since the colors on edges $x_iz$ are different, we may assume $c(x_2z)\neq c_{x_1}$ and $c(x_3z)\neq c_{x_1}$. Without loss of generality, $c(x_1z)\neq c_{x_2}.$ We have $c_{x_1}\neq c_{x_2}$, and so $x_1zx_2$ is a rainbow $x_1,x_2$-connector.
 For (b), from (a) we may assume that there are at most three colors on special edges incident to $z$. Let $x_1, x_2, x_3, x_4$ be such that $\alpha_z:=c(zx_1)=c(zx_2)$ and suppose $\alpha_z \notin \{c(zx_3), c(zx_4)\}.$ Without loss of generality $c_{x_3}\neq \alpha_z$, and $c_{x_1}\neq c(zx_3).$ Then $x_1zx_3$ is a rainbow $x_1,x_3$-connector.
\end{proof} 
%
To obtain rainbow cycles of bigger lengths we will use shorter rainbow cycles and extend them by wrapping around the sets $V_0',\dots, V_4'$.

\begin{defn}\label{extendable-def}A rainbow cycle is called {\it extendable} if it contains an edge $ux$ for some $x\in V_0'\cup \dots \cup V_4'$ and $u\in N^m(x)$.
\end{defn}
\begin{lemma}Let $l\in \{4,8,12,16\}$. Let $k\in \mathbb{Z}^+$  be such that $k\ll n$. If $H$  contains an extendable cycle of length $l$ then it contains an extendable cycle of length $l+5k$.
\end{lemma}
\begin{proof}We proceed by induction on $k\in \mathbb{Z}^+\cup \{0\}$. For the inductive step, suppose $C$ is an extendable cycle of length $l+5(k-1)$. Let $U:= c(E(C))$ and $W:=V(C)$.  Let $ux\in E(C)$ be such that $x\in V_i'$ and $u\in N^m(x)$. Let $v_0\in (N^+_D(u)\cap V_i')\setminus W$, $v_1\in (N^+_D(v_0) \cap V_{i+1}')\setminus W, \dots, v_4\in (N^+_D(v_3)\cap V_{i+4}'\cap N^m(x))\setminus W$ be such that $c(uv_0),c(v_0v_1),\dots, c(v_3v_4)$ are four different colors that are not in $U$. Removing the edge $ux$ from $C$ and adding $v_0v_1v_2v_3v_4$ gives an extendable cycle of length ${l+5k}$. 
\end{proof}
To obtain rainbow cycles we will have to consider different lengths $4,8,12,16$ but every cycle that we will obtain will be extendable.
There can be different types of exceptional edges in $H$. One type is special edges $vw$ with $v\in V_{i}'$ and $w\in V_{i-1}'\cup V_{i-1}''$, but there can also be edges from $V_i'$ to $V_{i+2}''$.

For $l \in \{1, \dots, 4\}$ let $S_i(l)$ be the set of vertices $v\in V_i'$ such that $v$ has neighbors in $V_{i+2}''$ of $l$ distinct colors that are different than $c_v$.
\begin{lemma}\label{ext-lem-S(l)}
Let $l \in \{ 1, \dots, 4\}$. If for some $i\in \{0,\dots, 4\}$, $|S_i(l)|\geq (2l{{+}}1)l$, then $H$ has an extendable cycle of length $4q$ for every $q=1,\dots, l$.
\end{lemma}
\begin{proof}
Without loss of generality suppose $|S_0(l)|\geq (2l{{+}}1)l$. 
For each $v\in S_0(l)$ we fix $l$ different colors on the edges between $v$ and $V_2''$ and consider the following auxiliary digraph on $S_0(l)$. For $x,y\in S_0(l)$ put the arc from $x$ to $y$ if $c_x$ is one of the colors on the edges from $y$ to $V_2''$. Then the maximum in-degree is at most $l$. Thus there is a vertex $v\in S_0(l)$ of the total degree (the sum of in- and out-degrees) at most $2l$. Consequently, the chromatic number of the digraph is at most $2l+1$ and so there is an independent set containing $l$ vertices, $x_1, \dots, x_l\in S_0(l)$. Then there is a matching $M$ in $H[V_0', V_{2}'']$ consisting of $l$ edges  $x_1y_1, \dots, x_ly_l$ with $x_j\in V_0'$, $y_j\in V_2''$  such that all colors $c_{x_1},\dots, c_{x_l}, c(x_1y_1),\dots, c(x_ly_l)$   are different. 
For $q=1$, select $u_1\in N^+_D(y_1)\cap V_3'$ so that $c(y_1u_1)\notin \{c_{x_1}, c(x_1y_1)\}.$ Now select $v_1\in N^m(x_1)\cap N^+_D(u_1)$ so that $c(u_1v_1)\notin \{c_{x_1}, c(x_1y_1), c(y_1u_1)\}$. Then $x_1y_1u_1v_1x_1$ is an extendable cycle of length four. 
For $1<q\leq l$,   pick $u_1, \dots, u_q$ and $v_1, \dots, v_q$ so that $x_1y_1u_1v_1x_2\dots u_qv_qx_1$ is an extendable cycle of length $4q$. 
\end{proof}

\begin{cor}\label{ext-cor-S(l)}
    If $H$ does not have an extendable cycle for some $l\in \{4, 8, 12, 16\}$ then for every $i$, $|S_i(4)| < 36$.
\end{cor}

We know further analyze special edges.
\begin{defn}
A special edge $vw$ with $v\in V_{i}'$ and $w\in V_{i-1}'$ is called {\it special of type 1} if $c_v=c_w$ and of {\it type 2} otherwise.
\end{defn}
Special edges of type 1 between $V_{i-1}', V_i'$ form a matching, because for every color $\alpha$ there is at most one vertex in $V_j'$ with $c_v=\alpha$. In addition, by Lemma \ref{ext-lem2} we may assume that for a vertex $w\in V_{i-1}'$ all special edges of type 2 that are incident to $w$ have color $c_w.$ We have the following lemma on the colors of special edges between $V_i'$ and $V_{i-1}'$. 

\begin{lemma}\label{ext-lem-smallset}
 Let $k\in \mathbb{Z}$ with $k\geq 2$, and $k\ll n$ and suppose there is no rainbow cycle in $H$ of length $2k$.
 For every $i$, there exist  sets $U_i\subseteq V_i'$ and $U_{i-1}\subseteq V_{i-1}'$ such that $|U_i|\geq (1-2\beta)|V_i'|,$ $|U_{i-1}|\geq (1-2\beta)|V_{i-1}'|,$ and all special edges in $H[U_i, U_{i-1}]$ are of type 1.
\end{lemma}
\begin{proof} Without loss of generality $i=1.$
For every color $\alpha$ there is at most one vertex $v\in V_{0}'$ such that $c_v=\alpha$, and so at most one vertex $v$ that is incident to a special edge of type two of color $\alpha$. Let $x_1\in V_1'$, $y_1\in V_{0}'$ be such that $x_1y_1$ is special of type 2.  
Suppose we have constructed a matching $x_1y_1, \dots, x_ly_l$ with $x_1,\dots, x_l\in V_1'$, $y_1,\dots, y_l\in V_1'$ consisting of special edges of type two that satisfies
the following:
\begin{itemize} \item All colors in $S:=\{c(x_1y_1), \dots, c(x_ly_l), c_{x_1}, \dots, c_{x_l}\}$ are different,
\item for $j=1, \dots, l-1$, $y_jx_{j+1}$ is an edge of color $c_{x_{j+1}}$, and
\item if $l\geq 2$, then $x_1y_l$ is an edge of color $c_{x_1}$.
\end{itemize}
\begin{figure}[h] 
    \centering
    \includegraphics[width=1\textwidth]{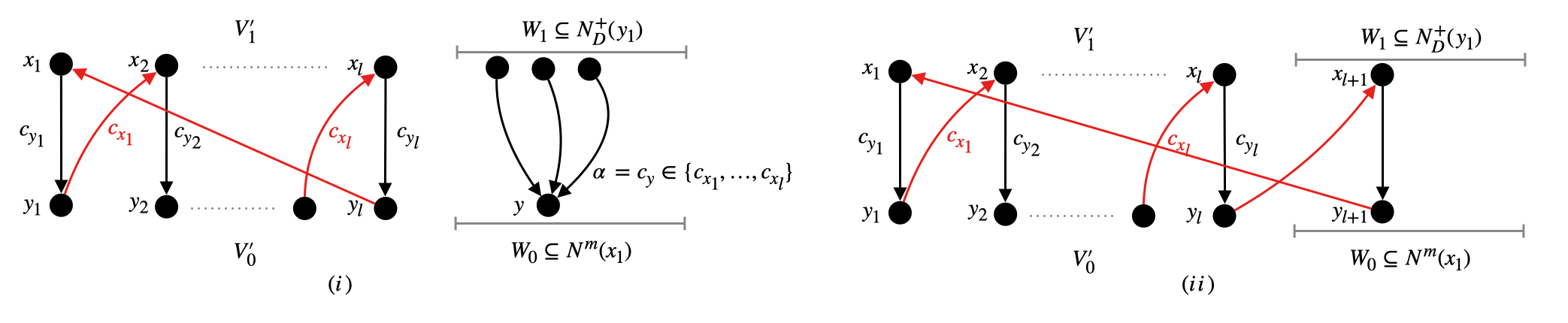} 
    \caption{Two cases in the proof of Lemma \ref{ext-lem-smallset}}
    \label{fig:Lem-4-19}
\end{figure}

We will argue that either we can add another edge to the matching or there exist sets $U_0, U_1$ from the statement of the lemma. Recall that by the comment above, for every $j$, $c(x_jy_j)=c_{y_j}$.
Let $z$ be a unique vertex in $N^+_D(y_l)\cap V_1'$ if one exists, such that $c(y_lz)=c_{y_l}$.
Let $W_1\subseteq V_1'$ be obtained from $N^+_D(y_l)\setminus \{x_1,\dots, x_l, z\}$ by deleting at most $l$ vertices $x$ such that for some $1\leq j\leq l$, $c_x=c(x_jy_j).$
Let $W_0:=N^m(x_1)\setminus \{y_2, \dots, y_l\}.$ If all special edges of type 2 in $H[W_0,W_1]$ have some color $\alpha$ from the set $S$, then this $\alpha \in \{c_{x_1},\dots ,c_{x_l}\}$ because special edges of color $\alpha$ form a star. We let $U_1:= W_1$ and let $U_0$ be obtained from $W_0$ by deleting at most $l$ vertices $z$ such that $c_z\in \{c_{x_1},\dots ,c_{x_l}\}.$  Then there are no special edges of type 2 between $U_0$ and $U_1$ and from (\ref{ext-Nm}), $|U_0|\geq (1-2\beta)|V_0'|$, $|U_1|\geq (1-2\beta)|V_1'|$. Thus we may assume there is a special edge $x_{l+1}y_{l+1}$ in $H[W_1, W_0]$ of type two  such that $c(x_{l+1}y_{l+1})\notin S$. Since $y_{l+1}\in N^m(x_1)$, $c(y_{l+1}x_1)= c_{x_1}.$  Since $c_{y_l}=c(x_ly_l)$, from the definition of $W_1$, we have $c_{x_{l+1}}\neq c_{y_l}$. Therefore, if $c(y_lx_{l+1})\neq c_{x_{l+1}}$, then $x_{l+1}y_l$ is special and we would have $c(y_lx_{l+1})=c_{y_l}$ (or a rainbow $C_{2k}$), but then  $x_{l+1}$ is not in $N^+_D(y_l)$ by the definition of $D$. 
Thus $c(y_lx_{l+1})=c_{x_{l+1}}$ verifying the second condition. 
If we can continue the process $k$ times then $H$ has a rainbow $C_{2k}$, and otherwise the sets $U_0, U_1$ exist.

\end{proof}
 Recall that an edge $vw$ in $D$ is called special if $v\in V_i'$ and $w\in V'_{i-1}\cup V_{i-1}''$, and a vertex is called special if it is incident to a special edge. We have the following simple observation.
\begin{lemma}\label{ext-C12-C4-lem1}
\begin{itemize}
\item[(a)] If there is a special vertex, then there is an extendable $C_{12}.$
\item[(b)] If for some $i$ there is $x\in V_i'$ and $v\in V_{i-1}'\cup V_{i-1}''$ such that $xv$ is special and $|N^+_D(v)\cap (V_{i-2}'\cup V_{i-2}'')|\geq 3$, then there is an extendable $C_4$.
\end{itemize}
\end{lemma}
\begin{proof} For (a) suppose $u\in V_{i-1}'\cup V_{i-1}''$ is special. Then $xu$ is special for some $x\in V_i'$. Let $y_0\in V_i'\cap N^+_D(u)$ be such that $xuy_0$ is a rainbow path that does not contain $c_x$. Add, one by one, nine different vertices $y_1,\dots, y_9$ so that for every $j$, $y_{j+1}\in N^+_D(y_j)\cap V'_{i+j}$, and $y_9\in N^m(x)$. Then $xuy_0\dots y_9x$ is an extendable $C_{12}$.

For (b), let $w\in N^+_D(v)\cap (V_{i-2}'\cup V_{i-2}'')$ be such that $c(vw)\notin \{c_x, c(vx)\}$. Then $|N^+_D(w)\cap V_{i-1}'|\geq \gamma |V_{i-1}'|$. Now let $z\in N^+_D(w)\cap N^m(x)$ be such that $c(wz)\notin \{c_x,c(xv), c(vw)\}$. Then $xvwzx$ is an extendable $C_4$.
 \end{proof}
\begin{lemma}\label{ext-C16-C8-lem1}
Let $x_1, x_2 \in V_1'$ and $v_1, v_2\in V_0'\cup V_0''$ be such that $x_1v_1, x_2v_2$ is a matching of size two, and  $c(x_1v_1), c(x_2v_2), c_{x_1}, c_{x_2}$ are four different colors.
\begin{itemize}
\item[(a)]  If $|N^+_D(v_1)\cap (V_{4}'\cup V_4'')|\geq 5$,  then there is an extendable $C_{16}$.
    \item[(b)] If for $j=1,2$, $|N^+_D(v_j)\cap (V_{4}'\cup V_4'')|\geq 7$,  then there is an extendable $C_8$.
    
\end{itemize}
\end{lemma}
\begin{proof}
Let $U:= \{c(x_1v_1), c(x_2v_2), c_{x_1}, c_{x_2}\}$. Let $z_1\in V_4'\cup V_4''$ be such that $v_1z_1\in H$ and $c(v_1z_1)\notin U$. Add $c(v_1z_1)$ to $U$ and let $w_1\in N^m(x_2)\cap N_D^+(z_1)$ be such that $c(z_1w_1)\notin U$. Add $c(x_1w_1)$ to $U$ and note that $x_1v_1z_1w_1x_2v_2$ is a rainbow path of length five that does not contain $c_{x_1}$. 
For (a), let $y_1\in N^+_D(v_2)\cap V_1'$ be such that $c(v_2y_1)\notin U$. This gives a rainbow path of length six that starts at $x_1$, ends in $V_1'$, and does not contain $c_{x_1}$. Now add nine additional vertices, one by one as in the previous proof, with the first one in $N^+_D(y_1)\cap V_2'$, and the last one in $N^m(x_1)\cap V_0'$. This gives an extendable $C_{16}.$
For (b), after selecting $z_1$, since $v_2$ has neighbors of at least seven  colors in  $V_{4}'\cup V_4''$, we can select $z_2\neq z_1$ such that $v_2z_2\in E(D)$ and $c(v_2z_2)\notin U$. Add $c(v_2z_2)$ to $U$ and find the path $x_1v_1z_1w_1x_2v_2$ as above.
 Now select $w_2\in N^+_D(z_2)\cap N^m(x_1)$ so that $c(w_2z_2)\notin U$ and $w_2\notin \{v_1, v_2, w_1\}$. This gives an extendable cycle of length $8$.
\end{proof}
Recall that $\delta^+(D)>n/5$ and from (\ref{ext-eta-beta-gamma}),  $\eta \ll \beta \ll \gamma$. Although we have a lot of information about vertices in $V_i'$, we need to prove stronger properties for $v\in V_i''$.
\begin{lemma}\label{ext-lem-beta-gamma}
 Let $i\in \{0, \dots, 4\}$ and let $v\in V_i'\cup V_i''$ be such that $|N^+_D(v)\cap V_{i-1}'|\leq 7$.
 \begin{itemize}
     \item[(a)]  If $N_H(v) \cap V_{i-1}'\neq \emptyset$,
     then $|N^+_D(v) \cap V_{i+1}'|\geq (1-{10}\beta)|V_{i+1}'|$.
     \item[(b)] $|N^+_D(v)
     \cap V_{i+1}'|\geq (1-2\gamma)|V_{i+1}'|$. 
 \end{itemize}
 \end{lemma}
 \begin{proof}
  We have $|N^+_D(v)\cap V_{i-1}'|\leq 7$, and, by Fact \ref{ext-fact2}(a), $v$ has no neighbors in $V_i'\cup V_{i+2}'$. Suppose first that $v$ has a neighbor $w$ in $V_{i-1}'.$ If $v$ has at least $2\beta |V_{i+3}'|$ neighbors in $V_{i+3}'$, then by (\ref{ext-Nm}), $v$ and $w$ have a common neighbor in  $V_{i+3}'$. Therefore, we have 
  $$|N^+_D(v)\cap V_{i+1}'| > n/5- 2\beta |V_{i+3}'|-7 - \abs*{\bigcup_{i=0}^4 V_i''} > (1- 10\beta)|V_{i+1}'|,$$
  since $\abs*{\bigcup_{i=0}^4 V_i''}\leq (\beta +\eta) n.$
 
 If $v$ has no neighbors in $V_{i-1}'$, then $|N^+_D(v)\cap V_{i+3}'|<\gamma |V_{i+3}'|$, as otherwise in view of Definition \ref{ext-def-V''} we would add $v$ to $V_{i+2}''$. Thus $$|N^+_D(v)\cap V_{i+1}'|>n/5- \gamma |V_{i+3}'|-7- \abs*{\bigcup_{i=0}^4 V_i''} > (1- 2\gamma)|V_{i+1}'|.$$ 
 \end{proof}
 \subsection{Rainbow cycles of length $4k$}\label{ext-subsec-2}
 We start with the following definition.
 \begin{defn}
 For a vertex $v\in V_i'$ and $l=1,2$ let $d^s_l(v)$ be the number number of colors on special edges $vw$ of type $l$   {such that $w\in V_{i-1}'\cup V_{i-1}''$}. \end{defn}
 Recall that special edges of type 1 form a matching in $E_H(V_{i-1}', V_i')$. As a result, for every $v\in V_{i}'$, $d^s_1(v)\leq 1$.  One of the main ideas used to address the extremal case is to look at the maximum special degree of type two and argue that either there exist extendable cycles or we get a contradiction with maximality. Specifically, let
 \begin{equation}\label{ext-eq-Mi}
 M_i:= \max_{X\subseteq V_i', |X|\geq |V_i'|/10}\min_{v\in X}d^s_2(v),
 \end{equation}
 and let
\begin{equation}\label{ext-eq-M}M:=\max_iM_i.\end{equation}

We  suppose that for some $l\in \{1, \dots, 4\}$ there is no extendable cycle of length $4l$. Recall from Corollary \ref{ext-cor-S(l)}, for every $i$, $|S_i(l)|< {36}$. 

The rest of the argument will use a tighter bound for the minimum color degree, and from now, we will assume $$\delta^+(D)=\delta^c(H)\geq \frac{n+6}{5}.$$
\begin{lemma}\label{ext-lem-size}
 For every $i=0, \dots, 4$, $|V_i'\cup V_i''|\leq n/5+ 4(M+4).$
\end{lemma}
\begin{proof}
Otherwise, there is $j$ such that $|V_j'\cup V_j''|< n/5 - (M+4)$. Let $v\in V_{j-1}'\setminus S_{j-1}(4)$.  We have $|N^+_D(v)\cap V_{j+1}''|\leq  3$ because $vw\in E(D)$ such that $c(vw)=c_v$ has $w\in V_{j-2}'$. There is exactly one vertex $w\in N^+_D(v)\cap V_{j-2}'$ such that $c(vw)=c_v$, $d^s_1(v)\leq 1$. By Fact \ref{ext-fact2}(b), 
$$d^s_2(v)> ((n+6)/5-5) - (n/5 -(M+4))>M.$$  Since $|V_{j-1}'\setminus S_{j-1}(4)|\geq |V_{j-1}'|/10$ this  contradicts the definition of $M$.
\end{proof}

\begin{lemma}\label{ext-set-S}
 Let $S:=\{v\in V_i'' : |N^-_D(v)\cap V_{i-1}'| \leq 2|V_{i-1}'|/3\}$. Then $|S|\leq 20M+80.$
\end{lemma}
\begin{proof}
Suppose $|S|> 20M+80$. We have $|E_D(V_{i-1}', S)|\leq 2|V_{i-1}'|\cdot  |S|/3$. Let $W:= \{v\in V_{i-1}' : |N^+_D(v)\cap S|<3|S|/4\}$. Then 
$$|E_D(V_{i-1}', S)|\geq (|V_{i-1}'|-|W|)\cdot \frac{3|S|}{4}, $$
and so $|W|\geq |V_{i-1}'|/9$. For every $v\in W$, by Lemma \ref{ext-lem-size}, 
$$|N^+_D(v)\cap (V_i'\cup V_i'')|\leq |V_{i}'\cup V_{i}''|- |S|/4< (n/5+4(M+4) -(5M+20) = n/5- M- 4.$$
Thus, for every $v\in W\setminus S_{i-1}(l),$ $d^s_2(v)> ((n+6)/5-5)-(n/5-M-4)>M$. Since $|W\setminus S_{i-1}(l)|\geq |V_{i-1}'|/10$, this contradicts the definition of $M$.
\end{proof}
We also have the following observation.
\begin{lemma}\label{ext-lem-Mdegs} Let $v\in V_{i}'\cup V_i''$ be such that $|N^+_D(v)\cap V_{i-1}'|\leq 7$ and $|N_H(v)\cap V_{i-1}'|\geq |V_{i-1}'|/2$. Then
\begin{itemize}
    \item[(a)] $|N^+_D(v)\cap (V_{i}'\cup V_i'')|=0$,
    \item[(b)] $|N^+_D(v)\cap (V_{i+2}'\cup V_{i+2}'')|\leq 20M+80$,
    \item[(c)] $|N^+_D(v)\cap (V_{i+3}'\cup V_{i+3}'')|\leq 20M+80$.
\end{itemize}
\end{lemma}
\begin{proof}
Without loss of generality $i=0$.  Let $v\in V_0'\cup V_0''$ be such that $|N^+_D(v)\cap V_{4}'|\leq 7$ and $N_H(v)\cap V_{4}'\neq \emptyset$. By Lemma \ref{ext-lem-beta-gamma}(a), \begin{equation}\label{eq-L-25}|N^+_D(v)\cap V_{1}'|\geq (1-10\beta)|V_1'|.\end{equation}

Since $\beta\ll \gamma$ and $H$ is triangle-free, $N_D^+(v)\cap (V_{0}'\cup V_0'')=\emptyset$. Thus (a) holds.

For (b), we have $N_D^+(v)\cap V_2'=\emptyset$  by Fact \ref{ext-fact2}(a). If $w\in N_D^+(v)\cap V_{2}''$, then for every $x\in N_D^+(v)\cap V_1'$, $xw\notin E(H)$. From (\ref{eq-L-25}) and the fact that $\beta \ll 1$, $|N^-_D(w)\cap V_2'|\leq |N_H(w)\cap V_2'|\leq  2|V_2'|/3$. Consequently, by Lemma \ref{ext-set-S}, $|N_D^+(v)\cap V_{2}''|\leq 20M+80.$

We will now show (c). If $w\in N^+_D(v)\cap V_{3}'$, then $v$ and $w$ share a neighbor in $V_{4}'$. Let $w\in N^+_D(v)\cap V_{3}''.$  Then, by (\ref{eq-L-25}), $|N^+_D(w)\cap V_1'|<10\beta |V_1'|$. If $|N^+_D(w)\cap V_{4}'|>|V_4'|/2$, then there is a triangle $vwx$ for some $x\in V_4'$. 

Thus, by Fact \ref{ext-fact2}(a), $|N^+_D(w)\cap V_2'|>|V_2'|/3+3$. If for four vertices $x_1, \dots, x_4\in N^+_D(w)\cap V_2'$, we also have $w\in N^+_D(x_i)$, then $c(wx_1),\dots, c(wx_4)$ are four different colors and, by the definition of $D$, for every $i$, $c(wx_i)\neq c_{x_i}$. Then, by Fact \ref{ext-fact4.5}(a) there is a rainbow $x,y$-connector for some $x,y\in V_2'$. Therefore, $|N^-_D(w) \cap V_2'| < 2|V_2'|/3$, and  by Lemma \ref{ext-set-S}, $|N^+_D(v)\cap V_3''|\leq 20M +80.$ 
\end{proof}

We will now establish an upper bound on $M$. Our first goal is to prove that $M\leq 1$. This will be done in the next two lemmas.
\begin{defn}
Let $B_i$ be the set of vertices $w\in V_i'\cup V_i''$ such that $w$ is incident to at least $44M+193$ special edges $vw\in E(D)$ of type 2 with $v\in V_{i+1}'.$\end{defn}

\begin{lemma}\label{ext-lem-B_i}
If $|B_i|\geq 20M+82$ then $H$ contains an extendable $C_{4l}$ for every $l\in [4]$.
\end{lemma}
\begin{proof}
The case $l=3$ follows from Lemma \ref{ext-C12-C4-lem1}(a). By Lemma \ref{ext-lem-path} we may assume there is no rainbow connector. Let $W:= \{w\in B_i: |N^-_D(w)\cap V_{i-1}'|\geq  2|V_{i-1}'|/3\}$. 
If $w\in B_i\cap V_i'$, then $w\in W$. Therefore, by Lemma \ref{ext-set-S}, $|W|= |B_i\cap V_i'|+ |W\cap V_i''|\geq 2.$
\begin{claim}\label{ext-cl>=7} If $w\in W$, then $|N^+_D(w) \cap (V_{i-1}'\cup V_{i-1}'')|\geq 7.$
\end{claim}
\begin{proof} Let $w\in W$ and suppose $|N^+_D(w)\cap (V_{i-1}'\cup V_{i-1}'')|\leq 6$. By Lemma \ref{ext-lem-Mdegs} and Fact \ref{ext-fact2}(a),
$$|N_D^+(w)\cap (V_{i-1}'\cup V_{i-1}'')|\geq d^c(w)- |V_{i+1}'\cup V_{i+1}''|- 40M-160+ (44M+193)-2,$$
because by Fact \ref{ext-fact4.5}(b) there can be at most two vertices $v\in N^+_D(w)\cap V_{i+1}'$ such that $vw$ is special of type 2.
In addition, by Lemma \ref{ext-lem-size}, $|V_{i+1}'\cup V_{i+1}''|\leq n/5 +4M+16.$
Therefore, 
$|N^+_D(w)\cap (V_{i-1}'\cup V_{i-1}'')|> 6$.
\end{proof}
Since $W \neq \emptyset$, from Claim \ref{ext-cl>=7} and  Lemma \ref{ext-C12-C4-lem1}(b), there is an extendable $C_4$.

We will now show that there is an extendable $C_{4l}$ for $l=2,4$. Let $w\in W$. By Fact \ref{ext-fact4.5}(b) there is a color $\alpha_w$ such that all special edges $vw$ with $v\in V_{i+1}'$ but at most one have color $\alpha_w$.  Suppose $x\in V_{i-1}'\cap N^-_D(w)$ is such that $c(xw)\neq {\alpha}_w$. Since $x\in N^-_D(w)$, by the definition of $D$, $c(xw)\neq c_x$. If $c(xw)\neq \alpha_w$, then we can choose $y\in V_{i+1}'$ so that $c(wy)=\alpha_w$ and $c_y\notin \{c(xw),c_x\}$. This gives a rainbow $x,y$-connector unless $c_x = \alpha_w$, and there is at most one $x$ like that. Therefore, we may assume that for every $x\in V_{i-1}'\cap N^-_D(w)$ but at most one, $c(xw)=\alpha_w$. Since $|W|\geq 2$, there are two different vertices $w_1,w_2\in W$. Then $|N^-_D(w_1)\cap N^-_D(w_2)\cap V_{i-1}'|\geq 3$ and if $x\in N^-_D(w_1)\cap N^-_D(w_2)\cap V_{i-1}'$ is not one of at most two exceptions, then $c(xw_1)=\alpha_{w_1}$ and $c(xw_2)=\alpha_{w_2}$. Consequently, by the construction of $D$, $\alpha_{w_1}\neq \alpha_{w_2}$. By Claim \ref{ext-cl>=7} and Lemma \ref{ext-C16-C8-lem1}, there is a rainbow $C_8$ and a  rainbow  $C_{16}$. 
\end{proof}
Next  lemma will be used to address extendable $C_4$ and $C_{16}.$
\begin{lemma}\label{ext-lem-M=1simple}
Suppose for some $i$ there is a set $X\subseteq V_i'$ such that $|X|\geq 500\gamma  |V_i'|$ and for every $x\in X$, $d^s_2(x)\geq \max\{M,1\}$. If for every $v\in V_{i-1}'\cup V_{i-1}''$ which is incident to a special edge of type 2 $vx\in E(D)$ with $x\in X$, $|N^+_D(v) \cap (V_{i-2}'\cup V_{i-2}'')|\leq 7$  then $|B_{i-1}|\geq 20 M+82.$
\end{lemma}
\begin{proof}
Without loss of generality $i=1$. Let $M^*:= \max\{M,1\}$. From Lemma \ref{ext-lem-smallset}  and (\ref{eq-V_i''}) there exist $W\subseteq V_0'\cup V_0''$ and $U\subseteq V_1'$ such that $|W|\leq 2\beta |V_0'|+|V_0''|\leq 8\beta |V_0'|$, $|U|\geq (1-2\beta)|V_1'|$, and for every $v\in U\cap X$, $vw$ is special of type 2 for at least $M^*$ vertices $w\in W$. 

Let $X^*:= U\cap X$ and $Z:=\{v\in V_0'': N_H(v)\cap V_4'=\emptyset\}$. We use the auxiliary constant $K=300$, and define $W'$ to be the set of vertices $w\in {W}$ that are incident to $K\cdot M^*$ special edges of type 2 { with one end in $X^\ast$}. Note that $K\cdot M^*> 44M+193$. Thus if $w\in W'$, then $w\in B_i$. We will now prove a lower bound for $|W'|.$

By Fact \ref{ext-fact4.5}, if $w\in W$ then all but at most one of special edges $vw$ of type 2 are of the same color. Therefore, since the edges in $D$ that begin at $w$ have distinct colors, there are at most two vertices $v\in X^*$ such that $vw$ is special of type 2. 

If $w\in W'\cap Z$, then by Lemma \ref{ext-lem-beta-gamma}(b), $w$ can be incident to at most $2\gamma|V_1'|+2<3\gamma |V_1'|$ vertices $v\in V_1'$ such that $vw$ is special of type 2. If $w\in W'\setminus Z$, then by Lemma \ref{ext-lem-beta-gamma}(a)  there are less than $11\beta |V_1'|$ vertices $v\in V_1'$ such that $wv$ is special of type 2. Therefore, counting special edges of type 2 between $X^*$ and $W$ gives
$$M^*|X^*| \leq 11\beta |V_1'|\cdot |W'\setminus Z|+ 3\gamma |V_1'|\cdot|W'\cap Z| +K\cdot M^*|W\setminus W'|.$$
Since by Lemma \ref{ext-set-S}, $|W'\cap Z|\leq |Z|\leq 100M^*$, and $|W\setminus W'|\leq |W|\leq 8\beta |V_0'|$,
$$|W'\setminus Z|\geq M^*\frac{|X^*|- 300\gamma |V_1'|-8\beta K|V_0'|}{11\beta |V_1'|}\geq 20 M + 82,$$
with a lot of room to spare, as $|X^*|\geq |X|-2\beta |V_1'|\geq (500\gamma-2\beta)|V_1'|$ and $\gamma \gg \beta.$
\end{proof}

\begin{cor}\label{ext-cor-C4}
    Suppose for some $i$ there is a set $X\subseteq V_i'$ such that $|X|\geq 500\gamma  |V_i'|$ and for every $x\in X$, $d^s_2(x)\geq \max\{M,1\}$. Then there is an extendable $C_4$ and an extendable $C_{12}$.
\end{cor}
\begin{proof}
An extendable $C_{12}$ follows immediately from Lemma \ref{ext-C12-C4-lem1}(a) because there exists a special vertex in $V_{i}'$.

Suppose there is no extendable $C_4$. By Lemma \ref{ext-C12-C4-lem1}(b), for every vertex $v\in V_{i-1}'\cup V_{i-1}''$ that is incident to a special edge $xv$ for  some $x
\in V_{i}'$,  $|N^+_D(v)\cap (V_{i-2}'\cup V_{i-2}'')|\leq 2$. From Lemma \ref{ext-lem-M=1simple}, $|B_i|\geq 20M+82$, and so, by Lemma \ref{ext-lem-B_i}, there is an extendable $C_4.$
\end{proof}
We will now prove the main lemma in this section.

\begin{lemma}\label{ext-lem-M=1}
 If for some $i$ there is a set $X\subseteq V_i'$ such that $|X|\geq 700\gamma |V_i'|$ and for every $x\in X$, $d^s_2(x)\geq \max\{M,2\}$, there is an extendable $C_{4l}$ for every $l \in [4]$.
\end{lemma}

\begin{proof}
In view of Corollary \ref{ext-cor-C4} we may assume indirectly that there is no extendable cycle of length $8$ or no extendable cycle of length $16$.  We will show that $|B_{i-1}| \geq 20M + 82$, and so the result follows by Lemma \ref{ext-lem-B_i}.

Without loss of generality $i=1$. 
Let $M^*:=\max\{M,2\}$. From Lemma \ref{ext-lem-smallset} there exist $W\subseteq V_0'\cup V_0''$ and $U\subseteq V_1'$ such that $|W|\leq 8\beta |V_0'|$, $|U|\geq (1-2\beta)|V_1'|$ and for every $v\in U\cap X$, $vw$ is special of type 2 for at least $M^*$  vertices $w\in W$. 
Let $X^*:= U\cap X$, $Z:=\{v\in V_0'': N_H(v)\cap V_4'=\emptyset\}$. Let $Y:=\{w\in V_0'
\cup V_0'':|N^+_D(w)\cap (V_4'\cup V_4'')|\geq 7\}.$

Suppose there exists $v\in Y$ such that for at least four vertices $w\in X^*$, $wv$ 
 is special of type 2. By Fact \ref{ext-fact4.5} there exists color $\alpha_v$ such that all but at most one of the special edges $wv$ with $w\in X^*$  have color $\alpha_v$.
Let $x_1,x_2\in X^*$ be two different vertices such that $vx_1$, $vx_2$ are special of type 2 and $c(vx_1)=c(vx_2)=\alpha_v$. Modify $H[W,X^*]$ as follows.
\begin{itemize}
    \item[(a)] Delete all special edges of type 2 incident to $v$.
    \item[(b)] Delete any remaining special edges of type 2 of color $\alpha_v$.
    \item[(c)] Remove from $X^*$ vertices $x_1,x_2$ and vertex $y$ such that $c_y=\alpha_v$ if it exists.
\end{itemize}
Let $\widetilde{W}$ be the set of vertices $w\in W$ that are incident to a special edge of type 2 in $H[W,X^*]$ after the modification.
\begin{claim}\label{cl1-ext-lem-M=1} There are at least $(M^*-1)(|X^*|-4)$ special edges of type 2 in $H[\widetilde{W}, X^*]$ after the modification.
\end{claim}
\begin{proof}
For every vertex $x\in X^*\setminus \{x_1,x_2,y\}$ we removed at most one edge incident to $x$ of color $\alpha_v$. In addition, we possibly removed one additional edge incident to $v$. Thus the number of special edges of type 2 after the modification is at least 
$$M^*(|X^*|-3)-(|X^*|-3)- 1\geq (M^*-1)(|X^*|-4),$$
because $M^*\geq 2$.
\end{proof}

\begin{claim} \label{cl2-ext-lem-M=1}$\widetilde{W}\cap Y=\emptyset$.
\end{claim}
\begin{proof}
Suppose $w\in \widetilde{W}\cap Y$ and $zw$ is special of type 2 after the modification. Then $w\neq v$ and $z\notin \{x_1,x_2, y\}.$ Since all special edges of color $\alpha_v$ are removed, $c(wz)\neq \alpha_v$.
Since $z\neq y$, $c_z\neq \alpha_v$. Without loss of generality, $c_{x_1}\neq c(wz)$. Thus $x_1v, zw$ is a matching and $c_{x_1},c_z, c(x_1v), c(wz)$ are four different colors. Then by Lemma \ref{ext-C16-C8-lem1} there is an extendable $C_l$ for $l=8, 16$.
\end{proof}
From Claim \ref{cl2-ext-lem-M=1} either every vertex $w$ which is incident to at least four special edges of type 2 in $H[W,X^*]$ is not in $Y$ or for every $w\in \widetilde{W}$, $w\notin Y$. In the former case, we let $\widetilde{W}$ to be the set of vertices incident to at least four special edges of type 2 prior to the modification. Then, in any case, $\widetilde{W}\cap Y=\emptyset.$ 

Let $K=300$ and let $W'$ be the set of vertices $w\in \widetilde{W}$ that are incident to $K\cdot (M^*-1)$ special edges of type 2 after the modification (or before if one is not made). Note that $K\cdot(M^*-1)> 44M^*+193$ because $M^*\geq 2$. In addition, $W'\cap Y=\emptyset$.

Now we proceed in a way that is analogous to the proof of Lemma \ref{ext-lem-M=1simple}. 
If $w\in W'\cap Z$, then by Lemma \ref{ext-lem-beta-gamma}(b), $w$ can be incident to less than $3\gamma |V_1'|$ vertices $v\in V_1'$ such that $vw$ is special of type 2. If $w\in W'\setminus Z$, then by Lemma \ref{ext-lem-beta-gamma}(a)  there are less than $11\beta |V_1'|$ vertices $v\in V_1'$ such that $wv$ is special of type 2. 
From Claim \ref{cl1-ext-lem-M=1} we have at least $(M^*-1)(|X^*|-4)$ special edges of type 2 in $H[\widetilde{W},X^*]$. Therefore,
$$(M^*-1)(|X^*|-4)\leq 11\beta |V_1'|\cdot |W'\setminus Z|+ 3\gamma |V_1'|\cdot|W'\cap Z| +K\cdot(M^*-1)|\widetilde{W}\setminus W'|.$$
By Lemma \ref{ext-set-S}, $|Z|\leq 60M+80\leq 60M^*+80\leq 200(M^*-1)$, and $|\widetilde{W}\setminus W'|\leq |W|\leq  8\beta |V_0'|$. Thus
$$|W'\setminus Z|\geq (M^*-1)\frac{|X^*|-4- 600\gamma |V_1'|-8\beta K|V_0'|}{11\beta |V_1'|}>(M^*-1)\frac{2\gamma}{\beta}\geq \frac{\gamma}{\beta}{M^*}\geq 20 M^* + 82,$$
using the bound for $|X^*|$, the fact that $\beta \ll \gamma$, and $M^*\geq 2$. If $w\in W'$ then $w\in B_{0}$ and so $|B_{0}|\geq 20M^*+82\geq 20 M+82$.
\end{proof}
Since $\gamma \ll 1,$ the following corollary follows immediately from Lemma \ref{ext-lem-M=1}.
\begin{cor}\label{ext-cor-M=1}
    If $M\geq 2$, then there is an extendable $C_{4l}$ for every $l\in \{1,2, 3,4\}.$
\end{cor}
Recall that we have $\delta^c(H)\geq \frac{n+6}{5}$. 
\begin{cor}\label{ext-cor-size}
If for some $i$, $|V_i'\cup V_i''|\geq n/5+ 20$, then  $H$ has an extendable $C_{4l}$ for every $l\in \{1,2, 3,4\}.$
\end{cor}
\begin{proof}
If $|V_i'\cup V_i''|\geq n/5+ 20$, then there is $j\neq i$ such that $|V_{j}'\cup V_j''|\leq n/5 - 5$. Then for every $v\in V_{j-1}'$, $v$ has neighbors of at least seven different colors that are not in $V_{j}'\cup V_j''$. Thus either $d^s_2(v)>1$ or $|N^+_D(v)\cap V_{j+1}''|\geq 4.$ Consequently, either $M>1$ and there is an extendable $C_{4l}$ by Corollary \ref{ext-cor-M=1}, or $|S_{j-1}(4)|\geq 36$ and there is an extendable $C_{4l}$ by Corollary \ref{ext-cor-S(l)}.
\end{proof}

Let $Z_{i+2}:= \{v\in V_{i+2}'' : \exists_{x\in V_i'} v\in N_D^+(x)\}$ and let $X_{i}:= \{v\in V_{i}' : N_H(v)\cap Z_{i+1}=\emptyset\}$. Note that if for some $i$, $Z_{i+2}\neq \emptyset$ then there is an extendable $C_4$. 
We have the following lemma.

\begin{lemma}\label{ext-lem-setZ}
If for some $l\in \{1, \dots, 4\}$, $H$ has no extendable $C_{4l}$ then the following holds.
\begin{itemize}
\item[(a)]For every $i$, $|Z_{i+2}|< { 50}$.
\item[(b)] For every $i$, there is $v_i\in X_i$ such that $|N^+_D(v_i)\cap (V_{i-1}'\cup V_{i-1}'')|\leq 3.$
\end{itemize}
\end{lemma}
\begin{proof}
For (a), let $x\in V_i'$ and suppose $z\in N^+_D(x) \cap Z_{i+2}.$ Then, from (\ref{eq-new}), Lemma \ref{lem-V'} and the fact that $\eta\ll \beta$, $|N^+_D(x)\cap V_{i+1}'|\geq (1-2\beta)|V_{i+1}'|$. Thus, since $H$ is triangle-free, $|N_H(z)\cap V_{i+1}'|< 2\beta |V_{i+1}'|$. Consequently,
\begin{equation}\label{eq-C16-Z}
|E_H(V_{i+1}', Z_{i+2})|< 2\beta |V_{i+1}'|\cdot |Z_{i+2}|.
\end{equation}
Let $U:=\{v\in V_{i+1}' : |N_H(v)\cap Z_{i+2}|\leq |Z_{i+2}|/2\}.$ From (\ref{eq-C16-Z}), $$|V'_{i+1}\setminus U|\cdot  |Z_{i+2}|/2 < |E_H(V_{i+1}', Z_{i+2})|< 2\beta |V_{i+1}'|\cdot |Z_{i+2}|.$$ Thus $|U|>(1-4\beta)|V_{i+1}'|>|V_{i+1}'|/2.$  By Corollary  \ref{ext-cor-size}, $|V_{i+2}'\cup V_{i+2}''|< n/5 +20$. Consequently, if $|Z_{i+2}|\geq 50$, then for every $u\in U$, $|N^+_D(u) \cap (V_{i+2}'\cup V_{i+2}'')|<(n/5+20)- { 25}$. Thus, either $M\geq 2$ and there is an extendable $C_{4l}$ by Corollary \ref{ext-cor-M=1} or $|S_{i+1}(4)|\geq 36$ and there is an extendable $C_{4l}$ by Corollary \ref{ext-cor-S(l)}.

For (b), from (\ref{eq-C16-Z}) and part (a), we have
$$|V_{i+1}'\setminus X_{i+1}|\leq |E_H(V_{i+1}', Z_{i+2})| \leq  2\beta |V_{i+1}'|\cdot |Z_{i+2}|< 100 \beta |V_{i+1}'|.$$
Thus \begin{equation}|X_{i+1}|>(1-100\beta)|V_{i+1}'|.
\end{equation} 
Let $v\in X_{i+1}$ and suppose $|N^+_D(v)\cap (V_{i}'\cup V_{i}'')|\geq 4$. Since $d^s_1(v)\leq 1$ and there is exactly one $w\in N^+_D(v)\cap (V_{i}'\cup V_{i}'')$ such that $c(vw)=c_v$, we have $d^s_2(v)\geq 2$, and so $M\geq 2$.
\end{proof}
\begin{lemma}\label{thm-general}
For every $k\in \mathbb{Z}^+$ there is $n_0\in \mathbb{Z}^+$ and $0<\eta<1$ such that for every $n\geq n_0$ the following holds. If $H$ is an edge-colored graph with $\delta^c(H)>n/5 +3$ that has no monochromatic path of length three and is such that $D_H$ is $\eta$-extremal, then $H$ contains a rainbow $C_{4k}$.
\end{lemma}
\begin{proof}
Suppose $H$ has no rainbow $C_{4k}$. Then $H$ has no extendable $C_{4l}$ for some $l\in \{1, 2, 3, 4\}$. From Lemma \ref{ext-lem-setZ}(b), for $i=0,\dots, 4$ there is $v_i\in X_i$ such that $|N^+_D(v_i)\cap (V_{i-1}'\cup V_{i-1}'')|\leq 3.$ Using Fact \ref{ext-fact2}(b) and the definition of $X_i$, $$n/5 +3<|N^+_D(v_i)|\leq 3+ |V_{i+1}'\cup V_{i+1}''|+|Z_{i+2}|-|Z_{i+1}|.$$ Therefore,
$$ n<\sum_{i=0}^4(|N^+_D(v_i)|-3)\leq \sum_{i=0}^4(|V_{i+1}'\cup V_{i+1}''|+|Z_{i+2}|)-\sum_{i=0}^4|Z_{i+1}|=n$$
\end{proof}
We can now quickly prove Theorem \ref{rainbowgeneral-thm}.\\
{\it Proof of Theorem \ref{rainbowgeneral-thm}.}
Fix $k$ and suppose $n$ is large enough. Let $H$ be an edge-colored graph of order $n$ that satisfies $\delta^c(H)> n/{5}+3$, has no rainbow $C_{4k}$, and is edge-minimal. Then $H$ has no monochromatic path of length three and $\delta^+(D_H)> n/{5}+3$.
Let $0<\eta< 1$ be such that Lemma \ref{thm-general} holds if $D_H$ is $\eta$-extremal. Since $H$ has no rainbow $C_{4k}$, by  Fact \ref{fact-about-C4}, the reduced digraph $\mathcal{R}_D$ has a directed cycle of length four, and so $D_H$ contains $\Omega(n^{4k})$ directed cycles of length $4k$. Thus, by Fact \ref{fact-di-cycles}, $H$ has a rainbow $C_{4k}$.$\Box$
 \subsection{Rainbow cycles of length four}\label{ext-subsec-3}

In this section, we will prove Theorem \ref{rainbowC4-thm}. Before the proof of the theorem, we address the extremal case. To simplify we let $V_i:=V_i'\cup V_i''$.


We have the following observation that was already indicated before.
\begin{lemma}For every $i$, if $v\in V_{i}'$, then $N_D^+(v)\subseteq V_{i-1}\cup V_{i+1}$ or there is a rainbow $C_4.$ \end{lemma}
\begin{proof}
Suppose that $vw\in E(D)$ and $w \notin V_{i-1}\cup V_{i+1}$. Then by Fact \ref{ext-fact2}(b), $w\in V_{i+2}''$, and so $c(vw)\neq c_v$. Since $|N^+_D(w)\cap V_{i+3}'|\geq \gamma |V_{i+3}'|$, we can select $x\in N^+_D(w)\cap V_{i+3}'$ such that $c(wx)\notin \{c_v, c(vw)\}$, and then select $y\in N^+_D(x)\cap N^m(v)$ to get a rainbow cycle of length four.
\end{proof}
By Corollary \ref{ext-cor-C4}, we may assume $M=0$. We also have $\delta^+(D)=\delta^c(H)= (n+7)/5= \lfloor n/5\rfloor+2$.

\begin{lemma}\label{lem}
    For every $i$, $|V_i|\geq \lfloor n/5\rfloor$.
\end{lemma}
\begin{proof}
  Since $\delta^c(H)\geq \lfloor n/5\rfloor+2$, if there is $j$ such that 
$|V_j|\leq \lfloor n/5\rfloor-1$, then for every $v\in V_{j-1}'$, $d^s_2(v)\geq \frac{n+7}{5}- (\lfloor n/5\rfloor-1) -1 -d^s_1(v)> 0$, since $d^s_1(v)\leq 1$. Thus, $M \geq 1$, a contradiction.
\end{proof}
By Lemma \ref{ext-lem-path}, we may assume that there is no rainbow connector.
\begin{lemma}\label{ext-plus-edges}
Let $i\in \{0, \dots, 4\}$, let $w\in V(H)$ be such that $|N^+_D(w)\cap V_i'|\geq 4$, and suppose there is no rainbow connector in $H$. Then all but at most three vertices $u\in N^+_D(w)\cap V_i'$ satisfy $c(wu)=c_u$. 
\end{lemma}
\begin{proof}
Suppose otherwise and recall that all edges in $D$ from $w$ to $N^+_D(w)$ have different colors. Let $u_1\in N^+_D(w)\cap V_i'$ be such that $c(wu_1)\neq c_{u_1}$. Then there exist $u_2, u_3 \in (N^+_D(w)\cap V_i') \setminus \{u_1\}$ such that $c_{u_1}\notin \{c(wu_2), c(wu_3)\}$. Without loss of generality, $c_{u_2}\neq c(wu_1)$. Then $u_1wu_2$ is a rainbow $u_1,u_2$-connector.
\end{proof}
We are now ready to prove the main lemma.
\begin{lemma}\label{lem-C4-extremal}
There is $n_0$  and $0<\eta<1$ such that for every $n\geq n_0$ the following holds. If  $H$ is an edge-colored graph on $n$ vertices with $\delta^c(H)\geq (n+7)/5$ that has no monochromatic path of length three and $D_H$ is $\eta$-extremal, then $H$ contains a rainbow cycle of length four.
\end{lemma}
\begin{proof}
Suppose $H$ has no rainbow cycle of length four.
We start with two  simple observations.
\begin{obs}\label{obs1}
   Suppose $xy$ is a special of type 1 with $x\in V_{j}'$ and $y\in V_{j-1}'$. Then for every edge $zy$ with $z\in V_{j-2}$, $c(zy)\in c(E_H(y, V_j))\cup \{c_y\}$. In particular, if $|V_{j}|\leq (n+1)/5$, then there cannot be a special edge $xy$ of type 1 with  $x\in V_j'$ and $y\in V_{j-1}'.$ 
\end{obs}
Otherwise,  for some $w\in N^+_D(z)\cap N^m(x)$,  $xyzw$ is a  rainbow $C_4$. 
\begin{obs}\label{obs2}
    Suppose there is a set $U\subseteq V_j'$ such that $|U|\geq |V_j'|/10$ and  for each vertex $u\in U$, $d^s_1(u)+d^s_2(u)\geq 1$. Then there is a special edge $xy$ of type 1 with $x\in U$ and $y\in V_{j-1}'$.
\end{obs}
Indeed, otherwise, since $\gamma \ll 1$, by Corollary \ref{ext-cor-C4}, there is  a rainbow $C_4$. \\

If there is an $i\in \{0,\dots, 4\}$ such that $|V_i|= |V_{i+1}|\leq  \frac{n+1}{5}$, then for every $u\in V_i'$, $d^s_1(u)+d^s_2(u)\geq 1$. Thus, by Obs. \ref{obs2}, there is a special edge $xy$  of type 1 with $x\in V_i'$, $y\in V_{i-1}'$ contradicting  Obs. \ref{obs1}.

Let $a:= n \bmod 5$. If  $a \in \{0, 1, 2, 4\}$, then since by Lemma \ref{lem}, for every $j$, $|V_j|\geq \lfloor n/5 \rfloor$  there is an $i$ such that $|V_i|= |V_{i+1}|\leq  \frac{n+1}{5}$. 

Suppose $a= 3$ and there is no index $i$ such that $|V_i|= |V_{i+1}|\leq  \frac{n+1}{5}$. 
Then, without loss of generality, $|V_0|= |V_2|= |V_4|= \frac{n+2}{5}$ and $|V_1|=|V_3|= \frac{n-3}{5}$.  

\begin{claim}\label{extC4-cl1}
    Let $z\in V_i$. Then $|N^-_D(z)\cap V_{i-1}'|\geq 3|V_{i-1}'|/4.$
\end{claim}
\begin{proof}
    We may assume $z\in V_i''$.  Let $U:= V_{i-1}'\setminus N^-_D(z)$ and suppose $|U|\geq |V_{i-1}'|/4.$ If $u\in U$, then either $uz\notin E(H)$ or $u\in N^+_D(z).$ By Lemma \ref{ext-plus-edges}, deleting at most three vertices, gives a set $U$ such that for every $u\in U$ either $uz\notin H$ or $c(uz)=c_u$. 
    
    If $|V_i|=\frac{n-3}{5}$ then for every $u\in U$, $d_s^1(u)+d_s^2(u)\geq 2$, contradicting $M=0$. Otherwise, $|V_i|=\frac{n+2}{5}$, $i\in \{0,2,4 \}$, and for every $u\in U,$ $d_s^1(u)+d_s^2(u)\geq 1$. Thus by Obs. \ref{obs2} there is a special edge $xy$ of type 1 with $x\in U$ and $y\in V_{i-2}'$.
    If $i=2,4$, then $|V_{i-1}|=\frac{n-3}{5}$ which contradicts Obs. \ref{obs1} (See Figure \ref{fig:Claim4-45}$(i)$). 
    
    If $i=0$, then $U\subseteq V_4'$. By Corollary \ref{ext-cor-C4}, for all but at most $500\gamma |V_0'|$ vertices $v\in V_0'$, $vx$ is special of type 1 for some $x\in V_4'$. Since special edges of type 1 form a matching, there is a special edge $vu$ of types 1 with $u\in U$ and $v\in V_0'$. Since $N^+_D(u)\cap V_0$ does not contain $z$, there is a vertex $y\in V_3$ such that $yu\in E(H)$ and $c(yu)\notin c(E_H(u, V_0)\cup \{c_u\}$, contradicting Obs. \ref{obs1}  (See Figure \ref{fig:Claim4-45}$(ii)$).
    \begin{figure}[h] 
    \centering
    \includegraphics[width=0.8\textwidth]{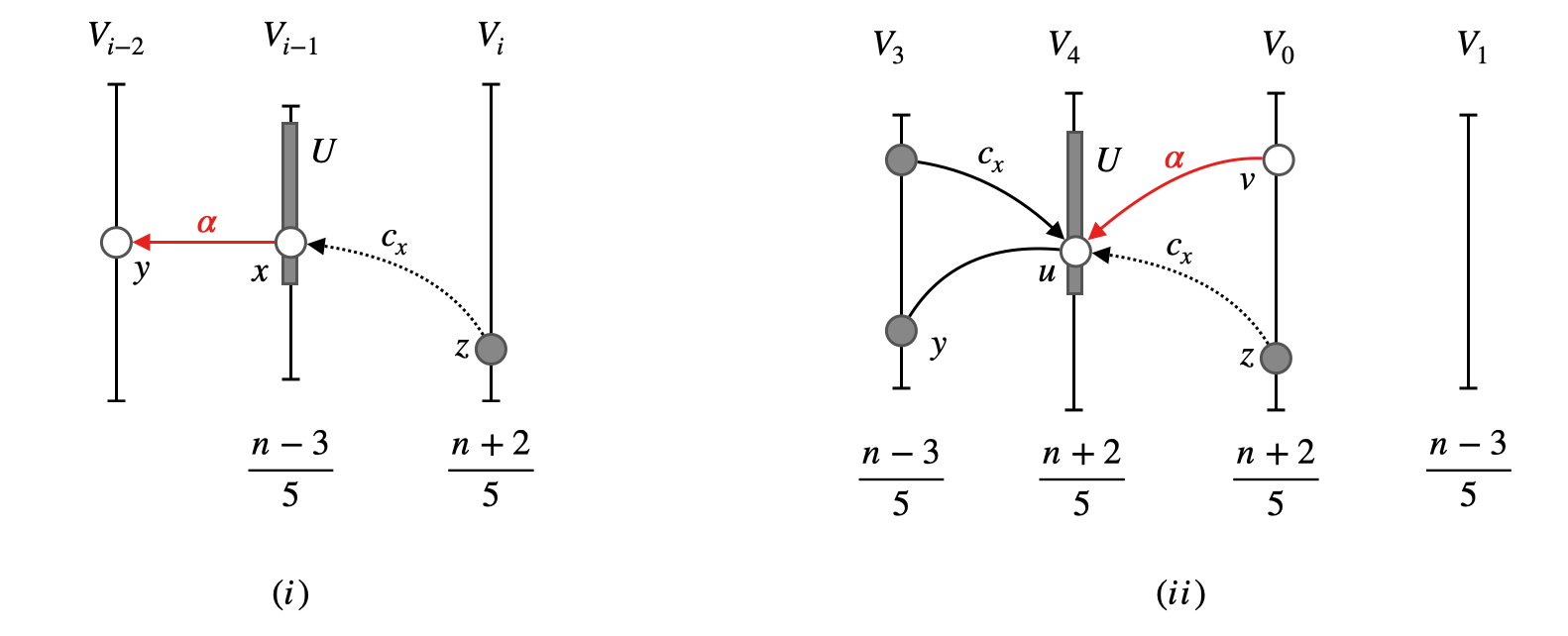} 
    \caption{Two cases in the proof of Claim \ref{extC4-cl1}}
    \label{fig:Claim4-45}
\end{figure}
\end{proof}
\begin{claim}\label{extC4-cl2}
    Let $z\in V_i$. Then $|N^+_D(z)\cap V_{i+1}'|>|V_{i+1}'|/2.$
\end{claim}
\begin{proof}
We may assume $z\in V_{i}''$. Since $H$ is triangle-free, and $\beta \ll \gamma$, $N_H(z)\cap (V_i'\cup V_{i+2}')=\emptyset$. By Claim \ref{extC4-cl1}, $N_H(z)\cap V_{i+3}'=\emptyset$ and $|N^+_D(z)\cap V_{i-1}'|<|V_{i-1}'|/4$. Finally,  by (\ref{eq-V_i''}), we have $|\bigcup V_j''|\leq 6\beta n$.
\end{proof}
\begin{claim}\label{extC4-cl3}
    Let $z\in V_i$. Then $N_H(z)\subseteq V_{i-1}\cup V_{i+1}.$
\end{claim}
\begin{proof}
If $u\in V_j$ for $j\notin \{i-1,i+1\}$, then by Claim \ref{extC4-cl1}  and Claim \ref{extC4-cl2} applied to both, $z$ and $u$, $u,z$ have a common neighbor.
\end{proof}
\begin{claim}\label{extC4-cl3}
    Let $z\in V_i''$. Then there is a color $c_z$ such that for all but at most one vertex $x\in N^-_D(z)\cap V_{i-1}'$, $c(xz)=c_z$.
\end{claim}
\begin{proof}
Suppose there exist two vertices  $x_1, x_2\in  N^-_D(z)\cap V_{i-1}'$  such that $c(x_1z)=c(x_2z)$. By the definition of $D$, we have  $c(x_iz)\neq c_{x_i}$. There is at most one vertex $y \in N^-_D(z)\cap V_{i-1}'$ such that $c_y= c(x_1z)$. For any other $v \in N^-_D(z)\cap V_{i-1}'$, $c_v\neq c(x_1z)$. Thus either there is a rainbow $x_1,v$-connector, or $c(vz)\in \{c_{x_1}, c(x_1z)\}$. If $c(vz)=c_{x_1}$, then there is a rainbow $x_2,v$-connector.  Since $v$ is any other vertex, there exist two vertices $x_1, x_2$ such that  $c(x_1z)=c(x_2z)$, and so, we must have for every $v$ but at most one exception, $c(vz)= c(x_1z).$
\end{proof}
Since there is no monochromatic path of length three, for any $x,y\in V_i$, $c_x\neq c_y.$

Recall that $|V_2|=\frac{n+2}{5}$ and $|V_1|=|V_3|=\frac{n-3}{5}$. Thus for every vertex $x\in V_2$, $x$ is incident to an edge $xu_x$ such that $c(xu_x)\neq c_x$ and $u_x\in V_1$. Since $|V_2|>|V_1|$, there is $w\in V_1$ such that $w=u_x$ and $w=u_y$ for two different vertices $x, y\in V_2$.  Without loss of generality, $c_w\neq c_x$. Then $c(wx)=c_w$ or there is a rainbow $C_4$. If $c(wy)=c_w$, then let $z\in V_0$ be such that $c(wz)\neq c_w$. Without loss of generality $c(wz)\neq c_x$ and we can select a vertex in  $N^m(x)\cap N^+_D(z)$ to get a rainbow $C_4$. If $c(wy)\neq c_w$, then we have $c_w=c_y$, since otherwise we get a rainbow $C_4$. Since $c(wx)=c_w$, there is $z\in V_0$ such that $c(wz)\notin \{c_y, c(wy)\}$. Now select a vertex in $N^+_D(z)\cap N^m(y)$ to get a rainbow $C_4$.
\end{proof}
{\it Proof of Theorem \ref{rainbowC4-thm}.}
Suppose $n$ is large enough and $H$ is an edge-colored graph of order $n$ that satisfies $\delta^c(H)\geq \frac{n+7}{5}$, has no rainbow $C_4$, and is edge-minimal. Then $H$ has no monochromatic path of length three and $\delta^+(D_H)\geq \frac{n+7}{5}$.
Set $0<\eta< 1$ so that Lemma \ref{lem-C4-extremal} holds if $D_H$ is $\eta$-extremal. Since $H$ has no rainbow $C_4$, by  Fact \ref{fact-about-C4}, the reduced digraph $\mathcal{R}_D$ has a directed cycle of length four, and so $D_H$ contains $\Omega(n^4)$ directed cycles of length four. Thus, by Fact \ref{fact-di-cycles}, $H$ has a rainbow $C_4$.$\Box$
\bibliographystyle{siam}
\bibliography{references}

\end{document}